\documentclass{article} 
\usepackage{iclr2026_conference,times}

\usepackage[utf8]{inputenc}
\usepackage[T1]{fontenc}
\usepackage{hyphsubst}
\usepackage{lmodern}

\usepackage[english]{babel}
\usepackage{mathtools}
\usepackage{bef_aras_nothm}
\usepackage{amsmath,amsthm,amssymb}
\usepackage{thmtools}
\usepackage{thm-restate}
\usepackage{amsfonts}
\usepackage{microtype}
\usepackage{float}
\usepackage{hyperref}
\usepackage{1710}
\usepackage{subfig}
\usepackage{array}
\usepackage{tikz}
\usepackage{verbatim}
\usepackage{afterpage}
\usepackage{braket}

\hypersetup{
     colorlinks   = true,
     citecolor    = blue
}

\usetikzlibrary{arrows,calc}

\tikzset{
  >=stealth',
  help lines/.style={dashed, thick},
  axis/.style={<->},
  important line/.style={thick},
  connection/.style={thick, dotted},
}

\newcounter{ArasCounter}

\newcounter{RumaCounter}

\newcounter{SeanCounter}

\newcommand{\revtwo}[1]{{#1}}

\makeatletter
\g@addto@macro{\thm@space@setup}{\thm@headpunct{}}
\makeatother

\numberwithin{equation}{section}

\allowdisplaybreaks

\newcommand{\Qvec}{\mathbf{Q}}

\newcommand{\abs}[1]{\left\lvert#1\right\rvert}
\newcommand{\dx}{\,{\rm dx}}

\newcommand{\norm}[1]{\left\lVert#1\right\rVert}

\begin{document}

\title{Deflation-PINNs: Learning Multiple Solutions for PDEs and Landau-de Gennes}


\author{Sean Disarò\thanks{seandisaro@gmail.com}, Ruma Rani Maity\thanks{Department of Mathematics, Indian Institute of Science, Bangalore, 560012, India, Email:  rumaranim@iisc.ac.in},
Aras Bacho\thanks{Department of Computing and Mathematical Sciences, California Institute of Technology, CA, USA, Email: bacho@caltech.edu}}

\date{}
\iclrfinalcopy
\maketitle

\begin{abstract}
Nonlinear Partial Differential Equations (PDEs) are ubiquitous in mathematical physics and engineering. Although Physics-Informed Neural Networks (PINNs) have emerged as a powerful tool for solving PDE problems, they typically struggle to identify multiple distinct solutions, since they are designed to find one solution at a time. To address this limitation, we introduce Deflation-PINNs, a novel framework that integrates a deflation loss with an architecture based on PINNs and Deep Operator Networks (DeepONets).

By incorporating a deflation term into the loss function, our method systematically forces the Deflation-PINN to seek and converge upon distinct finitely many solution branches. We provide \revtwo{theoretical results on the approximation capabilities} of our model and demonstrate the efficacy of Deflation-PINNs through numerical experiments on the Landau-de Gennes model of liquid crystals, a system renowned for its complex energy landscape and multiple equilibrium states\revtwo{, and on an Allen--Cahn benchmark whose solution set is provably known}. Our results show that Deflation-PINNs can successfully identify and characterize multiple distinct crystal structures\revtwo{: a single unsupervised run recovers all six stable states of the benchmark, each branch certified to lie in the basin of attraction of a different equilibrium, and the discovered branches are refined to percent-level accuracy by a purely neural Deflation--Deep-Ritz stage and to the accuracy of a mesh-converged reference by a classical solver that they initialize}
\end{abstract}



\section{Introduction}
Partial Differential Equations (PDEs) have a wide range of applications in many disciplines. Especially in physics, all major theories are formulated in the language of PDEs. However, not all PDEs can be solved analytically. Thus, it is important to improve and provide new numerical methods to find approximated solutions for PDEs, making it possible to apply these solutions in engineering and other scientific fields. With increasing computational power and with the rise of machine learning, there have been great ideas for solving PDEs with neural networks. Such methods include the Deep Ritz method \cite{DeepRitzPaper}, the Deep Galerkin method \cite{DeepGarlekinMethod}, Physics-Informed Neural Networks (PINNs) \cite{PINNOriginalPaper, PINNPaperPart1,PINNPaperPart2, lorenz24}, Variational PINNs (VPINNs) \cite{VPINNPaper} and Operator Learning \cite{FNO21,DeepONetsOriginal,bach25}. A notoriously hard-to-compute field in which it is necessary to solve PDEs numerically is the Landau-de Gennes (LdG) theory \cite{LDGTheoryPaper}, which is used to describe liquid crystals. In this context, some formulations of problems have a discrete number of solutions, which represent liquid crystal structures.

In this paper, we will explore a new method, which is based on
Physics-Informed Machine Learning (PIML) and which is designed to find solutions to an LdG problem, which was also investigated in \cite{MultistabilityApalachong}. Our framework is designed to be fully unsupervised and \revtwo{learns finitely many distinct solution branches of the LdG problem at once; for the benchmark considered here, the reported run recovers all six equilibrium states known from the FEM literature, and a certification-based restart protocol makes this completeness reproducible across random initializations.} It works by combining a PIML loss with a novel loss function, which ensures that the solutions found are distinct from each other. Furthermore, we establish universality by constructing our architecture as a special case of DeepONets—a neural network framework for learning operators between function spaces that is known to be universal \cite{DeepONetsOriginal, ChenChenApproxOP}.
With our approach, we leverage the strengths of PINN frameworks to capture multiple solution branches of ODEs/PDEs. As argued in \cite{multipleSolPINN}, such coarse approximations are particularly valuable when used as initializations for conventional numerical solvers, which can subsequently refine them into high-precision solutions. 

\subsection{Related Work}
\label{subsec:RelatedWork}
There has already been research which aims to find a finite number of solutions for a PDE. We give a brief summary of related work.

\textbf{HomPINNs.} In \cite{HomPINNs} the authors present the so-called Homotopy Physics-Informed Neural Networks (HomPINNs), which are similar to what we want. They are designed to allow for a discrete number of solutions for a PDE. Their approach, however, aims to fit data from observations and additionally incorporating a Physics-Informed loss in their total loss. Thus, this can be viewed as a semi-supervised setting.

\textbf{Learning Multiple Solutions from Random Initializations.} In \cite{multipleSolPINN} the authors explore a different setting, in which they train many conventional Physics-Informed Neural Networks (PINNs) and they initialize them with different random parameters in the hope that through this random initialization, the optimization process yields different solutions. Afterwards, the different solutions that were obtained were analyzed and compared to see which solution was found how often. The advantage of this approach is that it is easily scalable as one can implement many PINNs in parallel. However, since we cannot guarantee that the different solutions the PINNs find are distinct, this approach is also very prone to redundancy and may lead to non-optimal use of resources.

\textbf{Neural Networks for Nematic Liquid Crystals.} In \cite{NematicLiquidCrystalsNNPaper}, the authors solve a Landau-de Gennes  problem with neural networks using a loss function based on an energy functional. However, unlike our approach, this approach is not qualified to distinguish and find multiple solutions for an LdG problem.
In \cite{NatureLDGNeuralNetworks}, the authors use Convolutional Neural Networks (CNNs), to learn properties about liquid crystals. As the use of CNNs suggests, this work uses images of liquid crystals as inputs to predict properties like the order parameter of simulated nematic liquid
crystals.

\textbf{DeepONets.} \textit{Deep Operator Networks (DeepONets)} are a neural network framework for learning operators between function spaces. They were first introduced by \cite{DeepONetsOriginal} and have since been thoroughly studied in research. For a brief introduction to DeepONets, we refer to Appendix \ref{sec:DeepOnets} or to \cite{errorEstimationDeepONet} for a detailed survey on DeepONets.

\subsection{Overview}
In Section \ref{sec:methodology} we briefly recall the concepts of physics-informed machine learning. Then we proceed with presenting model architectures that are implemented with a hard constraint for Dirichlet conditions\revtwo{; the general radial-extrapolation construction behind our boundary-data extension is deferred to Appendix \ref{app:radial}}. In the following Subsection \ref{Subsection Deflation-PINNs}, we simplify the DeepONet architecture to fit our needs, and we argue that, for our purposes, we still maintain universality for the simplified model. More precisely, the simplified model will be defined in a way that it can only approximate a fixed number of functions. Moreover, we will introduce a new loss term, which makes the different solutions, which we want to approximate with our model, distinguishable. We call this new loss term \textbf{Deflation loss} and the total model will be called a \textbf{Deflation-PINN}. We will also provide a theoretical analysis on the approximation capabilities of our new model.\\
In the next section, we quickly recall the Landau-de Gennes model and introduce the concrete problem which we want to solve with our Deflation-PINN. \revtwo{We present the six solution branches discovered by a single unsupervised run, certify their basin membership against a mesh-converged finite-difference reference by a gradient-flow/Newton census, and refine them in two ways: by a purely neural Deflation--Deep-Ritz stage (Section \ref{subsec:DDR}) and by a classical solver that the network initializes, which reaches reference accuracy (Section \ref{subsec:refinement}). Section \ref{subsec:AllenCahn} adds a second benchmark, an Allen--Cahn problem on the unit disk whose solution set is provably complete; the choice of the deflation parameter $d_{min}$, a restart protocol with an unsupervised stopping rule, and convergence and deflation-variant studies are treated in Section \ref{subsec:refinement} and Appendix \ref{app:convergence}.} \revtwo{Finally, we close with a discussion and future directions}.

\section{Methodology}
\label{sec:methodology}
In this section, we discuss the model that we use in this paper, and for this we will recall fundamental concepts including Physics-Informed Machine Learning and neural networks with hard constraints.
\subsection{Physics-Informed Machine Learning}
The Idea of \textit{Physics-Informed Machine Learning (PIML)} was first presented in \cite{PINNOriginalPaper}.
The goal is to encode a PDE problem into a loss function, i.e let
\begin{equation}\label{DeRyckBVP}
    \begin{split}
        &\mathcal{D}[u](x)=0,~~~~\qquad\text{for }x\in\Omega\\
        &\mathcal{B}[u](s)=\psi(s),\qquad\text{for }s\in\partial\Omega,
    \end{split}
\end{equation}
where $\Omega\subset\mathbb{R}^d, d\in \mathbb{N}$ is compact and $\mathcal{D}, \mathcal{B}$ are differential operators and boundary operators, $u:\Omega\rightarrow\mathbb{R}^m,m\in\mathbb{N}$ is the solution of the PDE, $\psi:\partial\Omega\rightarrow\mathbb{R}^m$ specifies the boundary condition.
Then we can define the  \textit{generalization error} for (\ref{DeRyckBVP}) as
\begin{align}
\label{Generalization error}
\begin{split}
    \mathcal{E}_G(\theta):= \alpha_1 \cdot \int_{\Omega}|\mathcal{D}[u_\theta](x)|^2dx &+ \alpha_2 \cdot \int_{\partial\Omega}|\mathcal{B}[u_\theta](s)-\psi(s)|^2ds,
\end{split}
\end{align} where $u_\theta$ is a trainable model parameterized by $\theta$ in some parameter space $\Theta$, and $\alpha_1,~\alpha_2$ are hyperparameters.
The PIML framework seeks to minimize the difference between a model $u_{\theta}$ and the actual solution $u$ through the loss function above.
In other words, if the solution to (\ref{DeRyckBVP}) is unique, then the generalization error is zero if and only if $u_\theta = u$. This is due to the fundamental lemma of the calculus of variation.

When the model used is a neural network, then the respective neural network with PIML loss is usually referred to as a \textit{ Physics-Informed Neural Network (PINN)}.

To actually compute the generalization error $\mathcal{E}_G(\theta)$, we need to discretize it. The typical framework for this is to first discretize the differential operators in the first integral with automatic differentiation (AD) \cite{ADBook} and then discretize the integrals. For the discretization of the integrals, one can use e.g. a Riemann sum.
For more literature on PIML, we refer to \cite{ReviewPIML2021,wang2023expertsguidetrainingphysicsinformed,scholl26, SBBK23UPLL} and the references therein.

\revtwo{Note that instead of working with an $L^2$ error for $\mathcal{E}_G$, one can use any other $L^q$ norm, $1 \le q < \infty$, i.e.
\begin{align}
\begin{split}
    \mathcal{E}_{G,q} (\theta):= \alpha_1 \cdot \int_{\Omega}| \mathcal{D}[u_\theta](x)| ^q dx &+ \alpha_2 \cdot \int_{\partial\Omega}|\mathcal{B}[u_\theta](s)-\psi(s)|^q ds,
\end{split}
\end{align} where the appropriate choice of $q$ may depend on the regularity of the solution and of the residual.
}

\subsection{Hard Constraints For Boundary Conditions}
\label{Section Hard Constraints For Boundary Conditions}
Usually, when a loss function consists of multiple objectives, e.g., such as the two integrals in $\mathcal E _G$, it is not trivial how to choose the respective hyperparameter $\alpha _i$. This was discussed in more detail in \cite{PinnHyperParamTuning, hyperparamOptGeorgePINNs}. In such multi-objective loss functions, it is desirable to simplify the problem so that we have fewer objectives to approximate.\\
One way to achieve this in the PINN framework is to adapt the architecture of the neural network such that it automatically satisfies the boundary conditions, and thus the generalization error reduces to only $ 
\mathcal{E}_G(\theta):= \int_{\Omega}|\mathcal{D}[u_{\theta}](x)|^2dx.$
To ensure this, we can choose the following ansatz. 
\[
    \tilde f_\theta (x) := \omega (x) \cdot NN_\theta(x) + b(x) \approx u (x)
\]
where $\omega > 0$ on the interior $\text{int}(\Omega) $ and $\omega = 0$ on $\partial \Omega$. $b: \bar \Omega = \partial \Omega \cup \Omega \to \mathbb R$ should be at least continuous and chosen so that $b\vert _{\partial \Omega} = \psi$.
$NN_\theta(x)$ is a neural network and the approximation is true for a suitable choice of parameters $\theta$. Note that $\tilde f_\theta$ satisfies the exact Dirichlet condition independently of the choice of $\theta$.

The exact imposition of boundary constraints has been discussed extensively in literature; see, e.g., \cite{exactImpositionOfBDCon,DeepTheoryFuncConnections,hPINNPaper}. 

\revtwo{For the boundary-vanishing factor $\omega$ we adopt the $C^\infty$-construction of \cite{lu2021deepxde} on the unit square (Section \ref{sec:Experiments}); on the unit disk (Section \ref{subsec:AllenCahn}) the radial construction below yields $\omega(x) = 1 - \vert x \vert^2$ directly. For the extension $b$ of the (generally non-homogeneous) boundary data we use, in the discovery stage, a radial construction that applies to arbitrary bounded star-shaped domains (the Deflation--Deep-Ritz refinement of Section \ref{subsec:DDR} uses the harmonic extension of the same data instead); since only its two concrete instances are needed in this paper---on the unit square and on the unit disk---the general construction, including the radial boundary distance function and its continuity and smoothness properties, is deferred to Appendix \ref{app:radial}.}

\subsection{Deflation-PINNs}
\label{Subsection Deflation-PINNs}
There are PDE problems, which have a finite number of solutions $K>1$. For examples of such problems, we refer to the section where we present our numerical results. However, this kind of problem is not covered by the classical PINN framework, as it is only suitable for learning one solution at a time. Although the DeepONet framework can learn multiple solutions for a class of PDE problems, it is unlikely to be the most efficient choice, since it is built to learn infinitely many solutions in the form of an operator between function spaces. Thus, it may be seen as unnecessarily expensive if we only want finitely many solutions.
This section introduces \textit{Deflation-PINNs}, i.e. our new method for learning multiple solutions with neural networks.
The method consists of two new key points.
\begin{enumerate}
    \item We will introduce a neural network architecture to learn multiple solutions at the same time. The architecture will be a special case of the DeepONet architecture but designed in such a way that it only admits a finite number of solutions.
    \item We will introduce a loss function that ensures that we learn different solutions for the same problem.
\end{enumerate}

\subsubsection*{1) Architecture}
We start with (unstacked) DeepONets, which want to approximate an operator over a function space, i.e. they map $G:\mathcal A \times \mathbb{R}^d \to \mathbb{R}$, where $\mathcal{A}$ is a space of functions. The output then is of the form
$ \sum \limits _{i = 1} ^p \tau _ i (x) \cdot \beta _i(f) \approx G(f) (x)$
where $x \in \mathbb{R}^d$ and $f \in \mathcal{A}$. 
The $\tau _i $ are the i-th components of a neural network $ x \mapsto \tau(x) =  (\tau _1 (x), \dots , \tau _p (x))$. The $\beta_i$ use the evaluation of the function at previously fixed sensor points, that is,
take $( \tilde x_i ) _{i = 1} ^ N \in dom(f)$ sensor points, then 
\[f \mapsto (\beta _1 (f), \dots, \beta _p(f)) = (\beta _1 (f(\tilde x_1), \dots , f(\tilde x_N)), \dots, \beta _p(f(\tilde x_1), \dots , f(\tilde x_N)))\]
where $\beta = (\beta_1, \dots, \beta_p): \mathbb{ R}^N \to \mathbb{R}^p$ is a neural network.

However, in our setting, we want to learn only $K \in \mathbb N$ solutions that do not depend on an input function. Thus, we propose, instead of using a neural network $\beta_i: \mathbb R^N \to \mathbb R ^p $, to use only $K$ feature vectors of size $p$ (one feature vector for each solution), which encode the respective function, i.e., we want $\beta^k = (\beta^k _1 , \dots ,\beta ^k _p) \in  \mathbb{R}^ p$. These feature vectors can then be trained as part of the model's parameters. In total, our solution is of the form
\[
    \tilde G(k,x):= \sum \limits _{i = 1} ^p \tau _ i (x) \cdot \beta ^k _i \approx u_k (x),
\]
where $x\in dom(u_k)$ and $u_k$ for $k =1, \dots ,K$ are the $K$ solutions we want to approximate and $\tau_i$ are still neural networks. An illustration of our proposed architecture is given in Figure \ref{fig:DelfationPINNArchitecture}.
\begin{figure}[H]
\includegraphics[width = 1\textwidth]{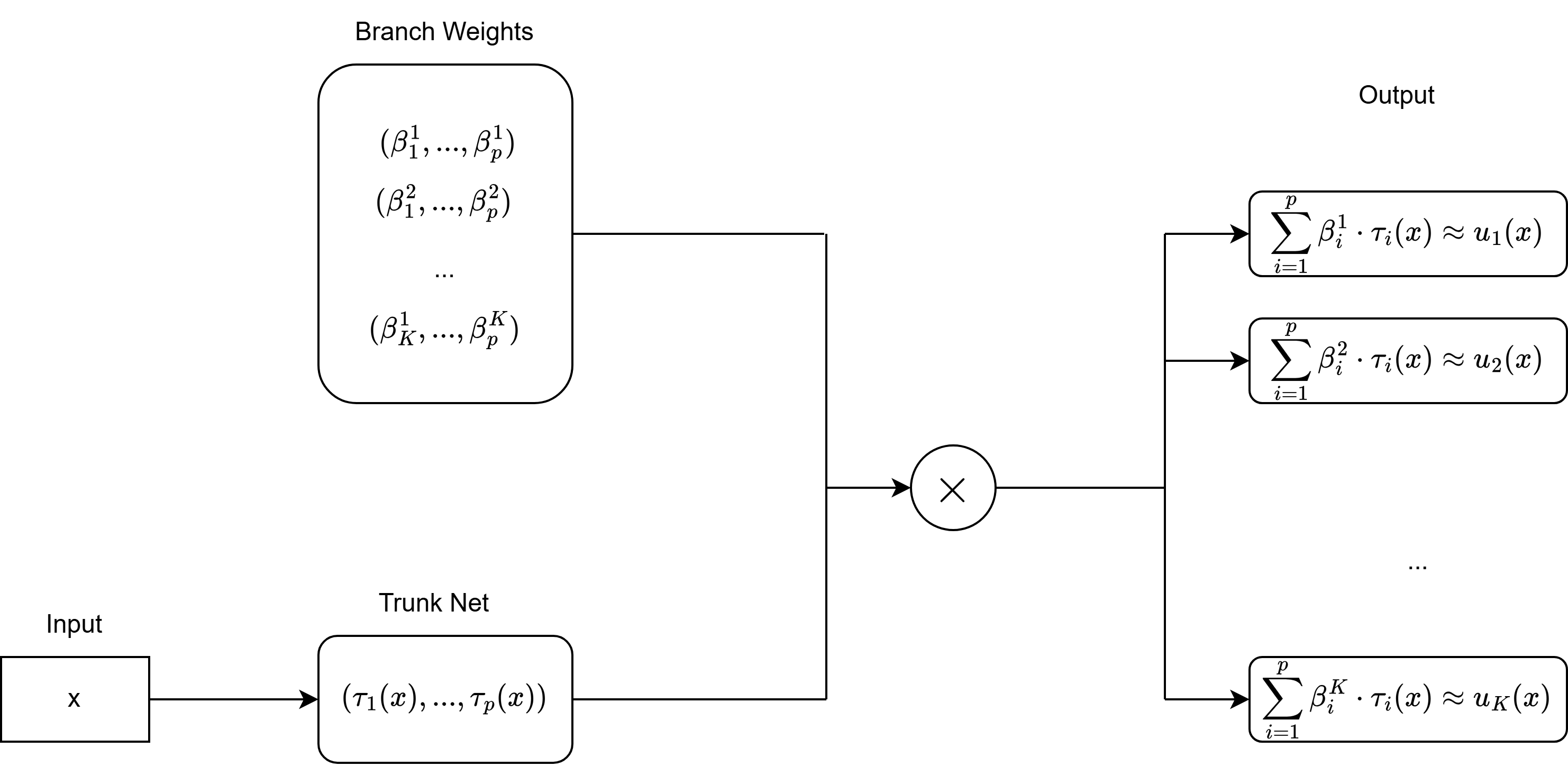}
\caption{Architecture of the Deflation-PINN. Unlike classical DeepONets, the branch weights $\beta_j^i$ are learned and do not use domain dependent sensor points. The trunk net takes the point $x\in \Omega$ at which we want to evaluate the respective solution. The output is computed by taking the dot product of the respective branch weights $(\beta_1^i, ..., \beta_p ^i)$ and the trunk net evaluations $(\tau_1 (x), ..., \tau _p (x) )$ to approximate the solution $u_i(x)$}
\label{fig:DelfationPINNArchitecture}
\end{figure}
The advantage of using fixed feature vectors $\beta ^k \in \mathbb R ^p$ instead of a neural network $\beta : \mathcal A \to \mathbb R^p$ is that we only need $ K \cdot p $ parameters to represent the $K$ functions. We call the $\beta ^k \in \mathbb R ^p$ \textbf{Branch Weights}.

If we were to use the DeepONet architecture, then we would need at least $S \cdot p + p$ parameters (one affine layer from the $S$ sensor values to $p$ features), where $S \in \mathbb N$ is the number of sensor points, which can easily be greater than the number of solutions, that is, $S > K$. Furthermore, the $K$ feature vector directly incorporates the solution structure such that it can be decoded by the architecture, since the feature vectors are trained to do so, which is in contrast to the fixed sensor points, which cannot be changed.

Note that since we have a finite number of solutions, which we want to approximate for our architecture, we still want to approximate a map of compact sets, and thus we can derive universality for this architecture as a special case of Theorem \ref{DeepONetChenChenThm} presented in Appendix \ref{sec:DeepOnets}.

\begin{cor}
    \label{corApproximationDeflationPINN}
    Suppose that $\sigma: \mathbb R\to \mathbb R$ is a continuous and non-polynomial activation function, $\mathbf{K} \subset \mathbb R^d$ is compact and $u_1, \dots, u_N \in C(\mathbf{K})$. 

    Then for all $\varepsilon >0$ there exist $ p\in \mathbb N$, $\zeta _ i \in \mathbb R$, $w_i \in  \mathbb R^d $ and $\beta_n^i \in \mathbb R$ ($i = 1, \dots, p$, $n = 1, \dots, N$) such that for all $n  \in \{1, \dots ,N\}$ and $y \in \mathbf{K}$ we have that
    \[
    \Big|  u_n(y)- \sum \limits _{i=1}^p \beta_n ^i \sigma(w_i \cdot y + \zeta_i) \Big| < \varepsilon.
    \]
\end{cor}
\begin{proof}
    Apply Theorem \ref{DeepONetChenChenThm} with $K_2 = \mathbf K$, an arbitrary compact $K_1$, a set $V = \{v_1, \dots, v_N\} \subset C(K_1)$ of $N$ distinct functions (finite, hence compact) and the operator $G(v_n) := u_n$, which is continuous because $V$ is finite. For each $n$, the branch factor $\sum_{i} c_i^k \sigma\big(\sum_{j} \xi_{ij}^k v_n(x_j) + \theta_i^k\big)$ multiplying $\sigma(w_k \cdot y + \zeta_k)$ in Theorem \ref{DeepONetChenChenThm} is a constant, which we call $\beta_n^k$; renaming the summation index $k$ to $i$ gives the claim.
\end{proof}

\revtwo{It is also worth noting that choosing a plain MLP as the trunk net is a somewhat arbitrary choice; for practical reasons one could integrate concepts such as Fourier features \cite{FourierFeatures}. In fact, we use Fourier features in Section \ref{subsec:DDR}.}

\subsubsection*{2) Deflation Loss}
We now have an architecture for our problem, but the question remains how to train this. With the PIML loss, we can learn true solutions. However,
we still need a loss, which tells the model that the $K \in \mathbb N$ solutions it learns (with the PIML loss) need to be distinct from one another. More precisely, the solutions $u_k$ that we want to approximate are pairwise at least $d_{min}$ apart, where $d_{min} \le \min \limits _ {i,j \in \{1,\dots, K\} , i\neq j} ||u_i - u_j||_{L^2}$ is a lower bound on their minimal separation. However, we note that, in general, we do not have access to this quantity. For that reason, we will also provide a strategy to estimate an admissible value from the training behavior alone (Appendix \ref{app:convergence}).
\\
Let $\tilde G$ with
\[
    \tilde G(k,x):= \sum \limits _{i = 1} ^p \tau _ i (x) \cdot \beta ^k _i \approx u_k (x)
\]
be our model.
We can rewrite the condition that the solutions be at least $d_{min}$ apart as the following loss function:
\[
    \mathcal L_{Def} (\tilde G) := {2 \over K \cdot (K-1)} \sum _ {i = 1}^K\sum _ {j = i+ 1} ^K \max\Big( 1 - {1 \over d_{min}}||\tilde G(i, \cdot ) - \tilde G(j, \cdot )||_{L^2} , 0 \Big),
\]
i.e., the respective approximated solutions $\tilde G(i, \cdot)$ have at least distance $d_{min}$ between each other in the $L^2$-norm if and only if $\mathcal L_{Def} (\tilde G) = 0$.

Thus, the total loss for training our model $ \tilde G$ can be given as
\[
    \mathcal L _ {total } (\tilde G) := \alpha\cdot \sum \limits _{i =1}^K  \mathcal E_G (\tilde G(i, \cdot)) + \beta\cdot \mathcal L_{Def} (\tilde G)
\]
where $\mathcal E_G$ is the PIML generalization error and for some hyperparameters $\alpha, \beta >0$.
We call this model $\tilde G$ together with the loss function $\mathcal L _ {total }$ a \textbf{Deflation-PINN}. $d_{min}$ is left as a hyperparameter which can be found empirically by testing different values or by some theoretical bound on the distance between solutions. Figure \ref{fig:DelfationPINNArchitecture} illustrates the architecture.

\revtwo{We remark that the hinge on output-field distances used here is one member of a family of admissible deflation terms, and by no means the unique choice. The distance may be measured in other norms; it may be measured in the network's \emph{intermediate representation}---in our architecture, directly between the branch vectors $\beta^k$, which encode the solutions---rather than in the output fields; the hinge may be replaced by scale-free pairwise repulsion terms such as $\sum_{i<j}(d_{min}/\Vert \tilde G(i,\cdot)-\tilde G(j,\cdot)\Vert)^{q}$, $q > 0$, closer in spirit to the deflation operators of Newton-based deflation \cite{Farrell2017} and to diversity losses in machine learning, which exert a (weak) separating force at every distance instead of vanishing beyond $d_{min}$; and the weights $\alpha,\beta$ as well as $d_{min}$ itself may be adapted during training. We use the hinge form throughout for its simplicity and interpretability; first empirical tests of the feature-space, adaptive-$d_{min}$ and repulsion variants are reported in Appendix \ref{app:convergence}.
Furthermore, we want to emphasize that the choice of the loss function $\mathcal E_G$ to encode the PDE's dynamics into the total loss function is an engineering choice and one could as well use concepts like the Deep-Ritz method to encode the PDE into the loss via an energy functional; see \cite{DeepRitzPaper}. We present a variation of the method described above in Section \ref{subsec:DDR}, where we use a Deep-Ritz loss.
}

\subsubsection*{3) Hard Constraints Approximation Theory}
Approximation theory for models with hard constraints is a topic which has already been discussed in literature; see e.g. 
\cite{exactImpositionOfBDCon}, where the authors use the framework of R-functions, to establish a strong theoretical framework.
Since we will also use hard constraints in the model for our experiments later, we want to invest a small section of this paper as well into justifying the use of such hard constraint models as described in Section \ref{Section Hard Constraints For Boundary Conditions}, via a novel elementary theorem. The theorem is formulated in the $\norm \cdot _\infty$-topology on the continuous functions, which makes it applicable in combination with Corollary \ref{corApproximationDeflationPINN}.
\begin{restatable}{thm}{theoremApproxmain}\label{thm For proving universality of imposed bc pinn}
    Assume $\Omega \subset \mathbb R^d$ is a bounded open domain. Let $g, h \in C( \Bar{\Omega} )$, s.t. $h \vert _{ \partial \Omega} = 0$ and $g > 0$ on $\Omega$, then there exists a sequence of functions $(f_n)_{n \in \mathbb N} \subset C(\Bar{\Omega})$, s.t.
    \[
    \norm{f_n g - h}_{\Omega} =\norm{f_n g - h}_{\infty, \Omega} \xrightarrow{n \to \infty} 0
    \]
\end{restatable}
We now apply the theorem above onto our Deflation-PINN architecture:
\begin{cor}
    Suppose that $\sigma: \mathbb R\to \mathbb R$ is a continuous and non-polynomial activation function, $\Omega \subset \mathbb R^d$ is a bounded open domain, $u_1, \dots, u_N \in  C(\bar \Omega)$, and $ g,h \in C(\bar \Omega)$ are such that $h\vert _ {\partial \Omega} = u_n \vert _ {\partial \Omega}$ for all $n \in \{ 1, \dots , N\}$, $g \vert _ {\partial \Omega}= 0 $ and $g > 0$ on $\Omega$.
    
    Then for all $\varepsilon >0$ there exist $ p\in \mathbb N$, $\zeta _ i \in \mathbb R$, $w_i \in  \mathbb R^d $ and $\beta_n^i \in \mathbb R$ such that for all $n  \in \{1, \dots ,N\}$ and $y \in \bar \Omega$ we have that
    \[
    \Big|  u_n(y) -\Big( h(y) + g(y)\cdot \sum \limits _{i=1}^p \beta_n ^i \sigma(w_i \cdot y + \zeta_i) \Big) \Big| < \varepsilon.
    \]
\end{cor}
\begin{proof}
    Fix $\varepsilon >0$ and set $\varepsilon' := \varepsilon / (1 + \norm{g}_{\infty, \Omega})$. Since $(u_n - h)\vert_{\partial \Omega} = 0$, Theorem \ref{thm For proving universality of imposed bc pinn} applied to $u_n - h$ and $g$ yields for each $n \in \{1, \dots ,N\}$ an $f_{\varepsilon, n} \in C(\bar \Omega)$ such that
    
    \[
    \norm{ \left( u_n - h\right) - g\cdot f_{\varepsilon, n}  }_{\Omega}  < \varepsilon'.
    \]
    By Corollary \ref{corApproximationDeflationPINN} (with $\mathbf K = \bar \Omega$, which is compact since $\Omega$ is bounded) there exist $p$, $\zeta_i$, $w_i$ and $\beta_n^i$ such that
    \[
        \Big|  f_{\varepsilon, n}(y)- \sum \limits _{i=1}^p \beta_n ^i \sigma(w_i \cdot y + \zeta_i) \Big| < \varepsilon' 
    \]
    for all $n$ and all $y \in \bar \Omega$.
    The triangle inequality then gives, for all $n$ and $y \in \bar\Omega$, $\big| u_n(y) - \big(h(y) + g(y) \sum_{i=1}^p \beta_n^i \sigma(w_i \cdot y + \zeta_i)\big)\big| \le \norm{(u_n - h) - g f_{\varepsilon,n}}_{\Omega} + \vert g(y) \vert \, \big| f_{\varepsilon,n}(y) - \sum_{i=1}^p \beta_n^i \sigma(w_i \cdot y + \zeta_i) \big| < \varepsilon' + \norm{g}_{\infty,\Omega}\, \varepsilon' = \varepsilon$.
\end{proof}

\section{Numerical Experiments}
\label{sec:Experiments}
In this section, we discuss numerical results on the benchmark example of the reduced two-dimensional LdG model of liquid crystal theory, and present the solutions obtained through our Deflation-PINN, together with their certification by a classical gradient-flow/Newton census, their refinement to reference accuracy---purely neural in Section \ref{subsec:DDR} and classical in Section \ref{subsec:refinement}---and a second benchmark with a provably complete solution set on the unit disk (Section \ref{subsec:AllenCahn}). The code can be found in \url{https://github.com/SeanDisaro/DeflationPINNs}.

\subsection{Landau-De Gennes Model}
Landau-de Gennes theory is a phenomenological variational theory that describes the state of nematic liquid crystals in terms of an order parameter, $\Qvec$ tensor. In two dimensions, $\Qvec$ tensor  is a symmetric traceless matrix in $\mathbb{R}^{2\times 2}, $ and is represented as a vector $\Qvec:=(Q_{11},Q_{12}).$ 
In the absence of surface effects  and external field, the reduced two dimensional LdG  \cite{MultistabilityApalachong} functional in the dimensionless form is given by
	\begin{align} \label{eq:2.1.1}
	\mathcal{E}(\Qvec) =\int_\Omega (\abs{\nabla \Qvec}^2 + 
	\epsilon^{-2}(\abs{\Qvec}^2  - 1)^2) \dx, 
	\end{align}
where $\Qvec\in \mathbf{H}^1(\Omega)$ and $\epsilon$ is a small parameter that depends on  the elastic constant, the bulk energy parameters, and the domain size.
We are concerned with the minimization of the functional $\mathcal{E}$ for the Lipschitz continuous boundary function $  \Qvec_b : \partial\Omega \rightarrow \mathbb{R}^2$ consistent with the experimentally imposed tangent boundary conditions. 
The energy formulation in (\ref{eq:2.1.1}) seeks $\Qvec \in \mathbf{H}^1(\Omega) $ such that 
	\begin{align}  \label{eq:LdGEL} 
	-\Delta \Qvec =2\epsilon^{-2}(1-\abs{\Qvec}^2)\Qvec \text{ in } \Omega \,\, \text{ and }\,\, \Qvec =  \Qvec_b \text{ on } \partial \Omega. 
	\end{align}
For our experiment, we use $\epsilon=0.02.$  
We solve the system (\ref{eq:LdGEL}) in $\Omega:=[0,1]\times [0,1]$ with 
 the Dirichlet boundary condition $ \Qvec_b$ constructed utilizing trapezoidal shape function $\textit{T}_d:[0,1]\rightarrow {\mathbb{R}}$ with  $d=3 \epsilon$ as
\begin{align} 
\label{eq:boundaryCondition}
	 \Qvec_b=
	&\begin{cases} 
	(\textit{T}_{d}(x),0) & \text{on}\,\,\,\, y=0 \,\,\,\,\text{and} \,\,\,\, y=1, \\
	(- \textit{T}_d(y),0)  & \text{on}\,\,\,\, x=0 \,\,\,\,  \text{and} \,\,\,\,  x=1, 
	\end{cases}
    \text{ and }\\	 
	\textit{T}_d(t)=
	&\begin{cases} 
	t/d, & 0 \leq t \leq d,  \\
	1, &  d \leq  t \leq 1- d, \\
	(1-t)/d, & 1- d \leq t \leq 1.
	\end{cases}
	\end{align}  
    Experimental and numerical investigations   suggest  that  there are  six stable solutions to this problem on a square domain, which have been computed using FEMs; see \cite{Tsakonas, MultistabilityApalachong,DGFEM}. Two
classes of stable experimentally observable configurations are reported: diagonal states (D1 and D2) in which the nematic directors are aligned along the square diagonals, and rotated states (R1, R2, R3, R4) in which  the  nematic directors rotate in $\pi$ radian angles across the square edges. 
 Some of these nematic equilibria are presented in  Figure \ref{fig:PlotSolAll}. 

\revtwo{\subsection{Solving LdG with Deflation-PINNs}}
\label{subsec:Solving LdG with DeflationPINNs}
This section is meant to give an overview of the architectural choices we used in our experiments. The Reproducibility Statement in the appendix includes more specific details about hyperparameter choices and hardware.
We use the model discussed in Section \ref{Subsection Deflation-PINNs} with a hard constraint for the Dirichlet boundary condition. The model is adapted  to incorporate  two outputs, i.e. the final model $\mathbf G = (G^1, G^2)$ is of the form
\begin{align*}
    G^1 :=&\revtwo{\hat G^1}(k, x) \cdot \omega (x) + \tilde Q_b(x)=\omega (x) \cdot \sum \limits _{i = 1} ^p \tau _ i (x) \cdot \beta _i ^k + \tilde Q_b(x) \approx u_k ^1 (x) \\
    G^2:=&\revtwo{\hat G^2}(k, x)\cdot \omega (x) ~~~~~~~~~~~=\omega (x) \cdot \sum \limits _{i = p+1} ^{2p} \tau _ i (x) \cdot \beta _i ^k ~~~~~~~~~\approx u_k ^2 (x),
\end{align*}
where $\beta ^k \in \mathbb R^{2p}$ for $k = 1, \dots, 6$ and $\tau: \mathbb R ^2 \to \mathbb R ^{2p}$ is a multi-layer perceptron (MLP) whose first $p$ outputs feed $G^1$ and whose last $p$ outputs feed $G^2$. \revtwo{$\hat G = (\hat G^1,\hat G^2)$ denotes the two-output Deflation-PINN core of Section \ref{Subsection Deflation-PINNs} before the hard constraint is applied; we write $\hat G$ instead of $\tilde G$ to keep this concrete two-output construction notationally distinct from the generic Deflation-PINN model $\tilde G$ of Section \ref{Subsection Deflation-PINNs}.}
$\omega: [ 0,1 ]^2 \to \mathbb R$ is such that $\omega > 0$ on $(0,1)^2$ and $\omega = 0$ on $\partial [ 0,1 ]^2$. $\tilde Q_b(x)$ is an extension of the trapezoidal function used in $\Qvec_b$, cf.\ \eqref{eq:boundaryCondition}, and is given by
\[
\tilde Q_b(x) \vert _{\partial [ 0,1 ]^2} = \Qvec_b ^1=
	\begin{cases} 
	\textit{T}_{d}(x) & \text{on}\,\,\,\, y=0 \,\,\,\,\text{and} \,\,\,\, y=1, \\
	- \textit{T}_d(y)  & \text{on}\,\,\,\, x=0 \,\,\,\,  \text{and} \,\,\,\,  x=1.
	\end{cases}
\]
Note that we did not add anything to the second output term of the model and multiplied it only by $\omega$. This is because we require the Dirichlet condition to be zero for the second term.
Concretely, we use the $C^\infty-$construction from \cite{lu2021deepxde} for $\omega$ and for $\tilde Q_b(x)$ we use the radial extension from Definition \ref{def radial Extrapolation} \revtwo{(Appendix \ref{app:radial})}. For this, we use the point $({1 \over 2},{1 \over 2})\in \Omega := [0,1]^2$ as a center for our star domain. 
The exact formula for the RBD function can be explicitly calculated and for this we refer to our implementation in our repository \url{https://github.com/SeanDisaro/DeflationPINNs}.
For the function $h: [0,1] \to [0,1]$ from Definition \ref{def radial Extrapolation} we simply use the function $h(x) = 1-x$, since this is very efficient to compute.\\
Regarding the PIML loss function, we only need to implement the dynamics of the model since we already satisfy the Dirichlet condition by construction. Thus, we can use equation (\ref{eq:LdGEL}) to formulate a generalization error of the form
\begin{align*}
   \mathcal E_G(\mathbf G) :=& \int _ {\Omega}\left | \epsilon ^2 \cdot \Delta  G^1 + 2\cdot (1- G^1 \cdot G^1  - G^2\cdot G^2 )\cdot  G^1 \right |^2 \dx  \\
   + &\int _ {\Omega}\left |\epsilon ^2 \cdot \Delta G^2  +  2\cdot (1- G^1 \cdot G^1  - G^2\cdot G^2 )\cdot G^2\right |^2\dx.
\end{align*}
\revtwo{Regarding the deflation loss, the definition in Section \ref{Subsection Deflation-PINNs} suggests computing the $L^2([0,1]^2, \mathbb R^2)$ norm of the full vector field $\mathbf G(i, \cdot) : \mathbb R ^2 \to \mathbb R ^2$. In the run reported below this recovers all six equilibrium states, and the learned branches remain well separated: the smallest pairwise distance between branches is $\approx 0.51$ in the full $L^2([0,1]^2, \mathbb R^2)$ norm, above the imposed $d_{min} = 0.4$ (for comparison, the true states are pairwise at least $1.086$ apart, cf.\ Section \ref{subsec:refinement}), so no two branches collapse onto one another.}\\

We compare our results with the results from \cite{FerroRMAMNN2021}, which were calculated using the finite element method. The plots can be seen in Figure \ref{fig:PlotSolAll}.
\revtwo{Furthermore, we compute the relative $L^2$ errors of the six learned branches against the mesh-converged reference solutions constructed in Section \ref{subsec:refinement}; they are reported in Table \ref{tab:relL2Differences}, together with the errors after the branches are refined by the classical procedure of Section \ref{subsec:refinement}.

For context, the errors of the discovery stage should not be compared one-to-one with accuracies reported by the related approaches of Section \ref{subsec:RelatedWork}: the HomPINNs of \cite{HomPINNs} operate in a semi-supervised regime---they fit observations sampled from the solutions themselves and report the accuracy of recovered equation parameters (of order $10^{-3}$--$10^{-5}$) rather than unsupervised field errors---while the random-initialization strategy of \cite{multipleSolPINN} trains many independent single-solution PINNs and selects distinct solutions afterwards. The discovery stage here is a \emph{single} fully unsupervised training run that must resolve six branches simultaneously. Section \ref{subsec:refinement} shows that this output determines all six states completely: refined by a classical solver that it initializes, every branch reaches reference accuracy.}

\begin{figure}[H]
    \centering
     \subfloat[$\Qvec^{NN}_{\text{D1}}$]{\includegraphics[width=0.25\linewidth]{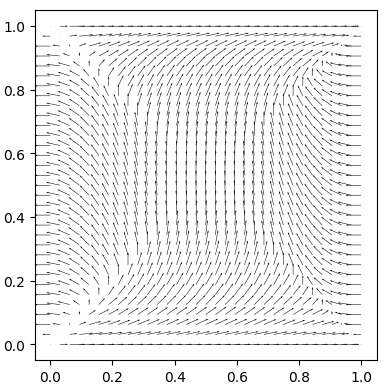}}
      \hspace{0.2 cm} \subfloat[$\Qvec^{NN}_{\text{R1}}$]{\includegraphics[width=0.25\linewidth]{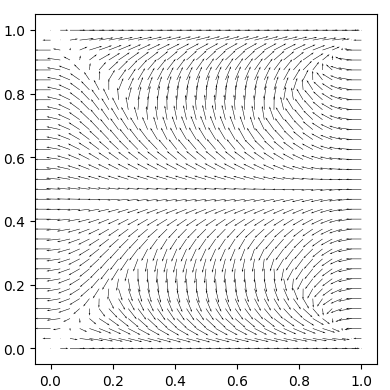}}
      \hspace{0.2 cm}   \subfloat[$\Qvec^{NN}_{\text{R3}}$]{\includegraphics[width=0.24\linewidth]{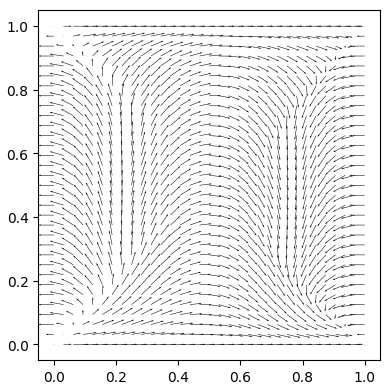}}
     \\    \subfloat[$\Qvec^{FEM}_{\text{D1}}$]{\includegraphics[width=0.25\linewidth]{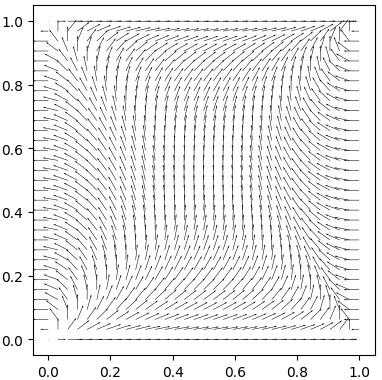}}  \hspace{0.2 cm} \subfloat[$\Qvec^{FEM}_{\text{R1}}$]{\includegraphics[width=0.25\linewidth]{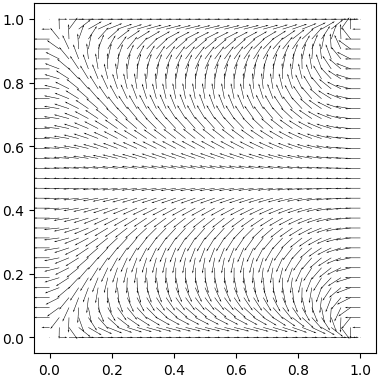}}
  \hspace{0.2 cm}    \subfloat[$\Qvec^{FEM}_{\text{R3}}$]{\includegraphics[width=0.25\linewidth]{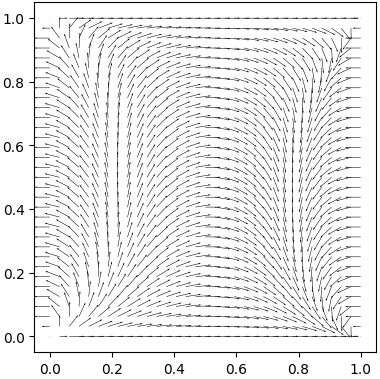}}
\end{figure}

\begin{figure}
    \centering
     \subfloat[$\Qvec^{NN}_{\text{D2}}$]{\includegraphics[width=0.25\linewidth]{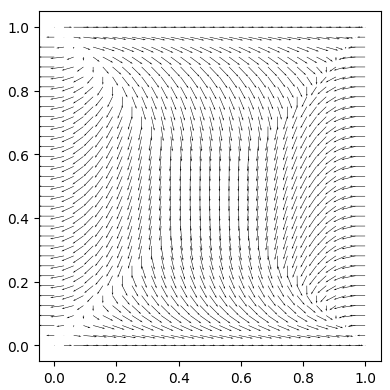}}
      \hspace{0.2 cm} \subfloat[$\Qvec^{NN}_{\text{R2}}$]{\includegraphics[width=0.25\linewidth]{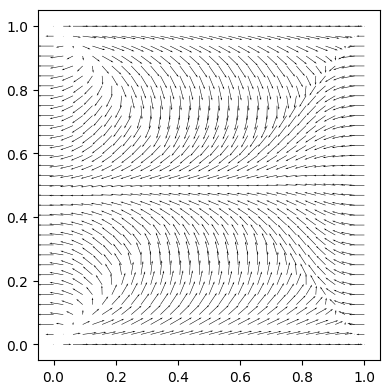}}
      \hspace{0.2 cm}   \subfloat[$\Qvec^{NN}_{\text{R4}}$]{\includegraphics[width=0.24\linewidth]{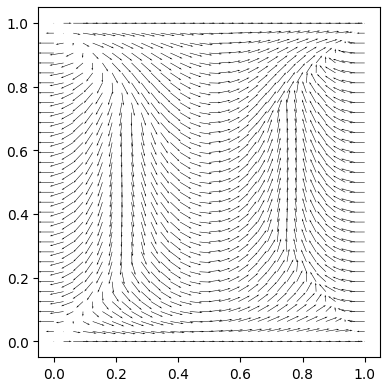}}
     \\    \subfloat[$\Qvec^{FEM}_{\text{D2}}$]{\includegraphics[width=0.25\linewidth]{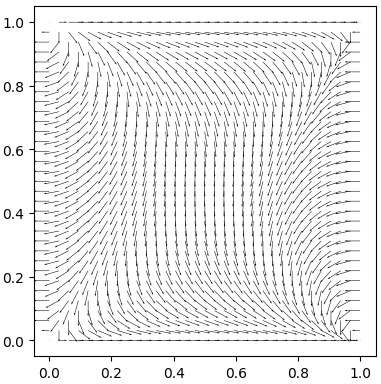}}  \hspace{0.2 cm} \subfloat[$\Qvec^{FEM}_{\text{R2}}$]{\includegraphics[width=0.25\linewidth]{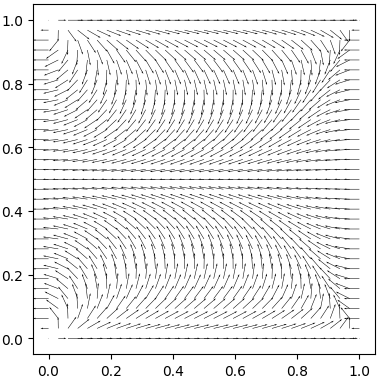}}
  \hspace{0.2 cm}    \subfloat[$\Qvec^{FEM}_{\text{R4}}$]{\includegraphics[width=0.25\linewidth]{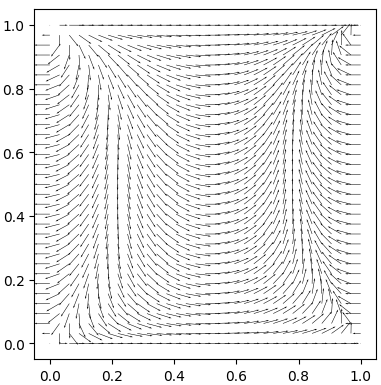}}
    \caption{\revtwo{
    Nematic director fields for the two diagonal states (D1, D2) and four
    rotated states (R1--R4) of the LdG problem \eqref{eq:LdGEL}.
    Rows one and three show the solutions $\Qvec^{NN}$ obtained by a single
    Deflation-PINN trained once with $K=6$ branches, while rows two and four
    show the corresponding FEM reference solutions $\Qvec^{FEM}$
    from \cite{FerroRMAMNN2021}.
    }}
    \label{fig:PlotSolAll}
\end{figure}

\begin{table}[H]
\begin{center}
\revtwo{\resizebox{\textwidth}{!}{\begin{tabular}{ | m{2.6em} | m{4.8em} | m{3.4em} | m{3.4em} | m{3.4em} | m{4.6em} | m{4.6em} | m{4.6em} |}
    \hline
     & Discovery & \multicolumn{3}{c|}{Deflation--Deep-Ritz} & \multicolumn{3}{c|}{Classical Refinement} \\
    \hline
     & rel.\ $L^2$ & rel.\ $L^2$ & rel.\ $L^\infty$ & rel.\ $H^1$ & rel.\ $L^2$ & rel.\ $L^\infty$ & rel.\ $H^1$ \\
    \hline
    D1 & $0.58$ & $0.013$ & $0.21$ & $0.16$ & $2.5\cdot 10^{-12}$ & $4.4\cdot 10^{-11}$ & $2.0\cdot 10^{-11}$\\
    \hline
    D2 & $0.59$ & $0.013$ & $0.21$ & $0.15$ & $2.5\cdot 10^{-12}$ & $4.4\cdot 10^{-11}$ & $2.0\cdot 10^{-11}$\\
    \hline
    R1 & $0.32$ & $0.013$ & $0.19$ & $0.14$ & $2.5\cdot 10^{-12}$ & $4.4\cdot 10^{-11}$ & $1.9\cdot 10^{-11}$\\
    \hline
    R2 & $0.30$ & $0.013$ & $0.19$ & $0.14$ & $2.5\cdot 10^{-12}$ & $4.4\cdot 10^{-11}$ & $1.9\cdot 10^{-11}$\\
    \hline
    R3 & $0.30$ & $0.012$ & $0.19$ & $0.14$ & $2.5\cdot 10^{-12}$ & $4.4\cdot 10^{-11}$ & $1.9\cdot 10^{-11}$\\
    \hline
    R4 & $0.30$ & $0.013$ & $0.20$ & $0.15$ & $2.5\cdot 10^{-12}$ & $4.4\cdot 10^{-11}$ & $1.9\cdot 10^{-11}$\\
    \hline
\end{tabular}}}
\end{center}
\caption{\revtwo{Errors of the six solution branches against the mesh-converged reference of Section \ref{subsec:refinement}, at the three stages of the pipeline: the discovery stage (a single unsupervised Deflation-PINN run, full vector-field deflation; relative $L^2$), the purely neural Deflation--Deep-Ritz refinement of Section \ref{subsec:DDR}, and the classical energy-descent/Newton refinement of Section \ref{subsec:refinement} initialized by the branches (each with relative $L^2$, relative $L^\infty$ normalized by $\max\vert\mathbf Q^{ref}\vert$, and relative $H^1$ via first-order finite differences of the error). The Deflation--Deep-Ritz block is evaluated on the $513^2$ reference grid; its $L^\infty$ and $H^1$ errors are dominated by the four defect cores of width $\sim\epsilon$, cf.\ Figure \ref{fig:DDRerr}. The classical-refinement block is evaluated against the discrete reference on the common $h=1/256$ grid: all three norms are at the rounding level of the Newton solve, i.e., the refined branches coincide with the reference solutions. The association of branches with states is not by proximity alone but by basin membership: flowing each branch under the LdG gradient flow yields a \emph{different} one of the six stable states (Section \ref{subsec:refinement}), so each branch is an approximation of exactly one state in the dynamically meaningful sense.}}

\label{tab:relL2Differences}
\end{table}

\revtwo{\subsection{A Mesh-Converged Reference and Refinement to Machine Precision}
\label{subsec:refinement}
\paragraph{Reference solutions.} To make the accuracy assessment independent of the resolution of any particular data export, we constructed mesh-converged reference solutions for all six states: a standard second-order finite-difference discretization of system \eqref{eq:LdGEL}, solved by damped Newton iteration on grids with $h = 1/256$ and $h=1/512$ and seeded by bilinear interpolation of the FEM data of \cite{FerroRMAMNN2021}. Every state converges in three Newton steps to a terminal residual of $5\cdot 10^{-12}$. The discrete energies are $79.455$ and $78.707$ (diagonal states) and $88.090$ and $87.342$ (rotated states) at $h = 1/256$ and $h = 1/512$, respectively; since the discrete energy functional (forward differences and a rectangle rule) is only first-order accurate, we use first-order Richardson extrapolation, $2E_{h/2} - E_h$, which yields $E_{D} \approx 77.96$ for the diagonal and $E_{R}\approx 86.59$ for the rotated states, matching the FEM energies computed in \cite{FerroRMAMNN2021} ($78.011$ and $86.648$) to $0.06\%$, which validates the reference.

\paragraph{Refinement of the discovered branches.} The Deflation-PINN is a \emph{discovery} device: a single unsupervised run must locate six distinct solution branches simultaneously. To determine what the discovered branches are worth quantitatively, we refine each branch by a classical two-step procedure that uses no information beyond the branch itself: (i) a semi-implicit discretization of the $L^2$-gradient flow of the LdG energy, $\partial_t \Qvec = \epsilon^2 \Delta \Qvec + 2\,(1-\vert \Qvec\vert^2)\Qvec$ (time step $0.1$, $3{,}000$ steps at $h=1/256$), which can only decrease the energy and therefore cannot terminate at unstable critical points, followed by (ii) damped Newton iteration, which supplies terminal quadratic convergence. The outcome is unambiguous: \emph{each of the six branches converges to a different one of the six stable states}---the two branches matched to the diagonal states flow to D1 and D2 (discrete energy $79.455$ at $h=1/256$) and the four remaining branches to R1--R4 (discrete energy $88.090$)---with terminal Newton residuals below $7\cdot 10^{-14}$ and relative $L^2$ distance $2.5\cdot 10^{-12}$ to the reference states (Table \ref{tab:relL2Differences}, right block). The deflation mechanism thus places every branch inside the basin of attraction of a distinct equilibrium. This realizes, for the multi-solution setting, exactly the pipeline advocated in \cite{multipleSolPINN}: the network provides coarse but correctly-basined approximations, and a classical solver initialized by them delivers high-precision solutions---here, all six of them, from one unsupervised training run.

\paragraph{The admissible range of $d_{min}$.} The reference solutions also quantify the choice of $d_{min}$ for this problem: the six stable states are pairwise at least $1.086$ apart in the discrete $L^2$ metric used by the deflation loss, and the certified unstable wall state below lies at distance $0.908$ from every stable state. Together with the Allen--Cahn observation of Section \ref{subsec:AllenCahn} (a $d_{min}$ \emph{above} the true minimal separation provably prevents convergence), this constrains the parameter from both sides: $d_{min}$ must lower-bound the minimal true separation, but in our experiments values well below that separation were already harmful---the hinge term then dominates early training, before the branches are near any solution, and prevents convergence altogether---so that on the LdG problem the operative upper edge is set by trainability ($\approx 0.5$) rather than by geometry ($1.086$). Both kinds of edge can be located automatically by a bisection on $d_{min}$ with the final residual as classifier (Appendix \ref{app:convergence}): on the Allen--Cahn benchmark the bisection locates the geometric edge without reference data (its final interval $[0.275, 0.288]$ lies within $6\%$ of the true minimal separation $0.272$), and on the LdG problem it places the practical trainability edge near $0.5$. The moderate value $d_{min} = 0.4 \ll 1.086$ is preferable in practice, at the price that pairwise separation alone cannot guarantee that two branches do not approximate the same solution: deflation constrains \emph{distances}, not \emph{basin membership}. Over-generating branches does not remove the duplication either: in a run with $K=11$ branches at $d_{min}=0.4$, nine branches certify under the gradient-flow/Newton census yet occupy only four distinct stable states (Appendix \ref{app:convergence}). The adaptive variants of Section \ref{Subsection Deflation-PINNs} (growing $d_{min}$ during training, scale-free repulsion terms) were tested for this reason; as reported in Appendix \ref{app:convergence}, they do not remove the duplication either, which is why completeness is delegated to the restart protocol described next.

\paragraph{Restarts with certification.} The census that certifies basin membership (gradient flow followed by Newton) costs minutes per branch on a CPU and uses no reference data, which turns the duplication problem into one that can be managed algorithmically: repeat the discovery run with fresh random initializations, certify the basins reached by every branch, retain each newly found basin, and stop once consecutive restarts contribute nothing new. In a repetition study at the published configuration (eight independent runs), individual runs certified between two and four distinct basins (mean $3.1$), with pronounced basin anisotropy---D2 was found by all eight runs, R1 and R2 by one run each---while the certified \emph{union} covered all six states after five runs and remained complete over the remaining three (Appendix \ref{app:convergence}, Table \ref{tab:restarts}). Completeness of a single run is therefore not guaranteed (the run reported in Table \ref{tab:relL2Differences} does certify all six basins), but the certified union over a handful of restarts reached completeness with an unsupervised stopping rule; we regard this restart-with-certification protocol, rather than any single run, as the operative completeness statement of the method.

\paragraph{A caveat on residual-based post-training.} We caution against refining the branches by continued minimization of the squared PDE residual alone. The pointwise reaction residual $\vert 2(1-\vert\Qvec\vert^2)\Qvec\vert^2 = 4\rho^2(1-\rho^2)^2$, $\rho = \vert \Qvec\vert$, vanishes both on the physical well $\rho = 1$ and at the isotropic state $\rho = 0$, with watershed at $\rho = 1/\sqrt 3$: wherever the scalar order parameter of a branch dips below this threshold, squared-residual descent drives it toward the unphysical isotropic well rather than the physical one. In our experiments, residual-only post-training accordingly degraded the amplitude-deficient branches and drifted them toward wall-type configurations: fields with interior lines on which $\vert\Qvec\vert \approx 0$ and strongly elevated energies. Such wall states are genuine \emph{unstable} solutions of \eqref{eq:LdGEL}: since the boundary data has $Q_{12} = 0$, the set $\{Q_{12}\equiv 0\}$ is an invariant subspace, and by Newton iteration with $\epsilon$-continuation inside this subspace we certified an X-shaped wall solution at $\epsilon = 0.02$ with energy $366.07$ and terminal residual below $10^{-14}$; the residual-refined branch that resisted all further refinement is precisely the one closest to this state. Unstable order-reconstruction states of this type---with the WORS cross state of \cite{KraljMajumdar2014} in the small-$\epsilon$ regime as the paradigm---are well documented for this model \cite{KraljMajumdar2014,CanevariMajumdarWang2017}; indeed, the solution landscape of \eqref{eq:LdGEL} on the square contains, in addition to the six stable states, a whole hierarchy of such saddle points \cite{HanSolutionLandscape2021}. Energy-descent refinement is immune to both failure modes, which is why it is used above. We note the flip side of this observation: precisely because residual least squares is stability-blind, deflated residual minimization with \emph{more} branches than stable states is a natural candidate for computing the unstable part of the solution landscape; a systematic study is left to future work.}

\revtwo{\subsection{A Deflation--Deep-Ritz Method}
\label{subsec:DDR}
Since problem \eqref{eq:LdGEL} is the Euler--Lagrange equation of the energy \eqref{eq:2.1.1}, the deflation construction of Section \ref{Subsection Deflation-PINNs} combines equally naturally with a Deep-Ritz-type loss \cite{DeepRitzPaper}: replacing the PIML residual term by the (discretized) energy itself,
\[
    \mathcal L_{DDR}(\mathbf G) := \alpha \cdot \sum_{k=1}^{K} \mathcal E\big(\mathbf G(k,\cdot)\big) \;+\; \beta \cdot \mathcal L_{Def}(\mathbf G),
\]
yields what we call a \emph{Deflation--Deep-Ritz} method. Its training can only decrease the branch energies while the deflation term keeps the branches pairwise separated, and by the stability discussion of Section \ref{subsec:refinement} it cannot terminate at unstable critical points: zero-deflation local minimizers of $\mathcal L_{DDR}$ are, by construction, \emph{stable} states that are pairwise at least $d_{min}$ apart. The two loss forms are thus complementary: the residual form is stability-blind and sees every solution of \eqref{eq:LdGEL}, whereas the Deep-Ritz form targets precisely the stable part of the solution set. One might hope to dispense with the discovery stage entirely and minimize $\mathcal L_{DDR}$ from random initialization; this fails in a structured way---all branches collapse into the two lowest-energy basins, with the deflation constraint saturated exactly at $d_{min}$---and is documented in Appendix \ref{app:convergence}. The residual-based discovery stage is what provides basin coverage.

We employ this method as a purely neural refinement of the discovery stage, keeping the pipeline unsupervised end to end: the six discovered branches are first projected onto six independent smooth networks (random Fourier features matched to the core scale $\epsilon$, with the harmonic extension of the boundary data as the hard constraint) by fitting the discovery output itself---a self-distillation step that uses no reference data---and $\mathcal L_{DDR}$ is then minimized with an unbiased quadrature of the energy: Adam on a $192^2$ composite-midpoint grid, double-precision L-BFGS on a $384^2$ midpoint grid, and a final double-precision L-BFGS continuation in which the midpoint rule ($O(h^2)$) is replaced by Simpson quadrature ($O(h^4)$) on a $321^2$ grid, with $d_{min} = 0.8$ throughout. (The value $0.8$ exceeds the trainability edge $\approx 0.5$ of the discovery stage; it is admissible here because the distilled branches start pairwise $\gtrsim 1$ apart, so the hinge is inactive from the first step---the edge is an early-training phenomenon of randomly initialized branches. For the same reason the branch/trunk coupling of the discovery architecture is not needed at this stage, and each branch is represented by its own network.) Full details are given in the Reproducibility Statement. The two-stage optimization is essential and consistent with standard PINN practice: a first-order method alone does not reach high accuracy---near a minimizer the landscape is flat and Adam stalls---so Adam is used only to approach the minimizers, and the final descent is performed by the quasi-Newton L-BFGS method (strong Wolfe line search) in double precision; in our runs L-BFGS continued to reduce both the loss and the error where Adam had plateaued.

The outcome is reported in the Deflation--Deep-Ritz block of Table \ref{tab:relL2Differences}: all six branches reach relative $L^2$ errors between $1.2\%$ and $1.3\%$ against the mesh-converged reference---uniformly across the diagonal and rotated states---with branch energies $80.0$--$80.2$ for the diagonal states and $88.3$--$88.5$ for the rotated ones (reference values $78.0$ and $86.6$). Figure \ref{fig:DDRSolutions} shows the six fields: the scalar order parameter satisfies $\vert\Qvec\vert \approx 1$ throughout the bulk, the only defects are the four corner cores, no interior order-reconstruction walls are present, and the director topology of every state is exact. The remaining percent-level gap is a quadrature effect: the minimizer of the \emph{discretized} Ritz functional lies within the quadrature error of the true minimizer, and the corner cores of width $\epsilon$ dominate that error. This diagnosis is confirmed experimentally by the quadrature continuation itself: with the training budget otherwise unchanged, raising the quadrature order from midpoint ($O(h^2)$, errors $1.9$--$2.4\%$ after the first two stages) to Simpson ($O(h^4)$) reduced the error to $1.2$--$1.3\%$---the error moves with the quadrature, not with the optimizer or the architecture.

The $L^\infty$ and $H^1$ columns of Table \ref{tab:relL2Differences} complement the relative $L^2$ errors, and Figure \ref{fig:DDRerr} shows contour maps of the pointwise error $\vert \mathbf Q^{NN}-\mathbf Q^{ref}\vert$. The localization of the error is instructive. After the Deflation--Deep-Ritz refinement, $59$--$77\%$ of the squared error is concentrated in the four corner regions of side $0.1$ containing the defect cores of width $\sim\epsilon$ (before the Simpson continuation this fraction was $92$--$94\%$: the higher-order quadrature reduces precisely the core error), the boundary layer contributes $16$--$26\%$, and the bulk $7$--$15\%$; the $L^\infty$ error is attained \emph{at} the cores. By contrast, the error of the discovery stage is a \emph{bulk} phenomenon (Figure \ref{fig:discerr}: $65$--$81\%$ of its squared error lies at distance $\ge 0.1$ from the boundary), consistent with the amplitude-deficiency mechanism discussed in Section \ref{subsec:refinement}. The two stages are thus limited in entirely different ways, and only the core-scale error of the refined fields is a genuine resolution effect.

\begin{figure}[H]
\centering

\includegraphics[width=\textwidth]{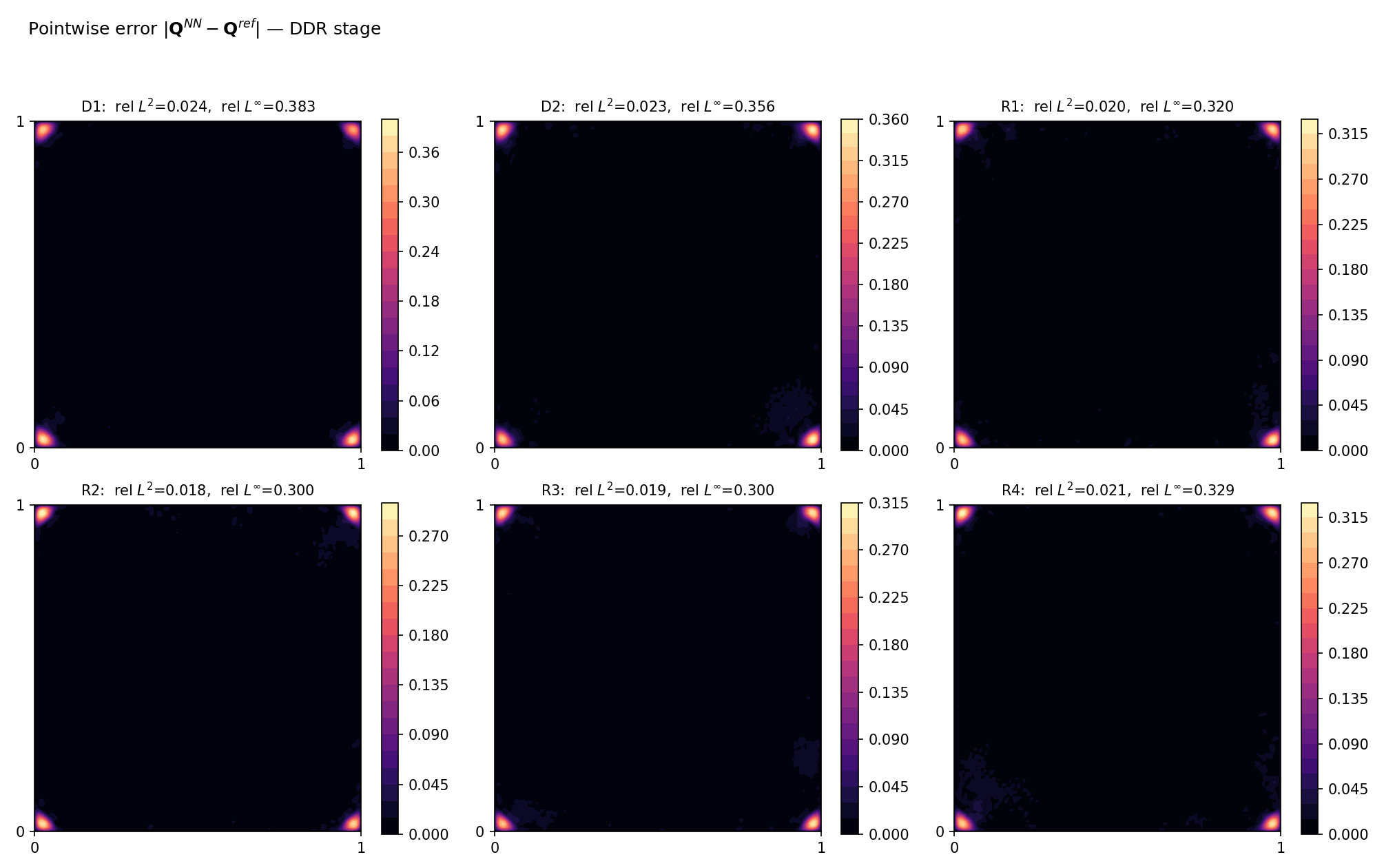}
\caption{\revtwo{Contour maps of the pointwise error $\vert \mathbf Q^{NN}-\mathbf Q^{ref}\vert$ of the six Deflation--Deep-Ritz solutions. The error is concentrated at the four corner defect cores ($59$--$77\%$ of the squared error within the four corner squares of side $0.1$); away from the cores the fields are essentially exact.}}
\label{fig:DDRerr}

\vspace{0.5em}

\includegraphics[width=\textwidth]{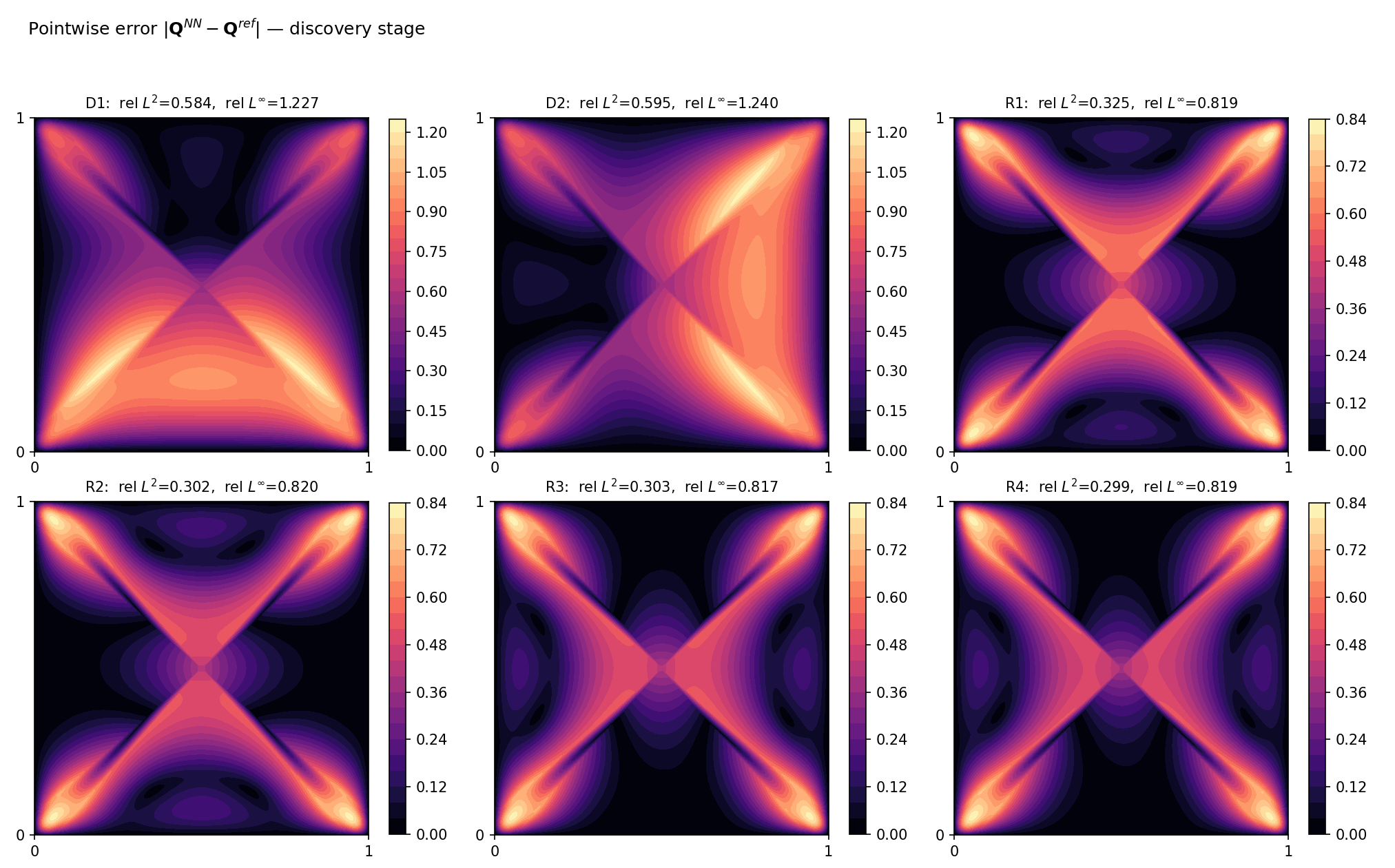}
\caption{\revtwo{Contour maps of the pointwise error $\vert \mathbf Q^{NN}-\mathbf Q^{ref}\vert$ of the six branches of the \emph{discovery} stage (Table \ref{tab:relL2Differences}, second column), on the same layout as Figure \ref{fig:DDRerr}. In contrast to the refined solutions, the error is a bulk phenomenon---broad amplitude-deficit regions in the interior rather than localized defect-core spots---which is the visual signature of the watershed mechanism of Section \ref{subsec:refinement}: the pointwise reaction residual tolerates fields with $\vert\Qvec\vert$ below its physical value, so residual training leaves the scalar order parameter under-saturated in the bulk while the director topology is already correct.}}
\label{fig:discerr}

\end{figure}

\begin{figure}[H]
\centering
\includegraphics[width=\textwidth]{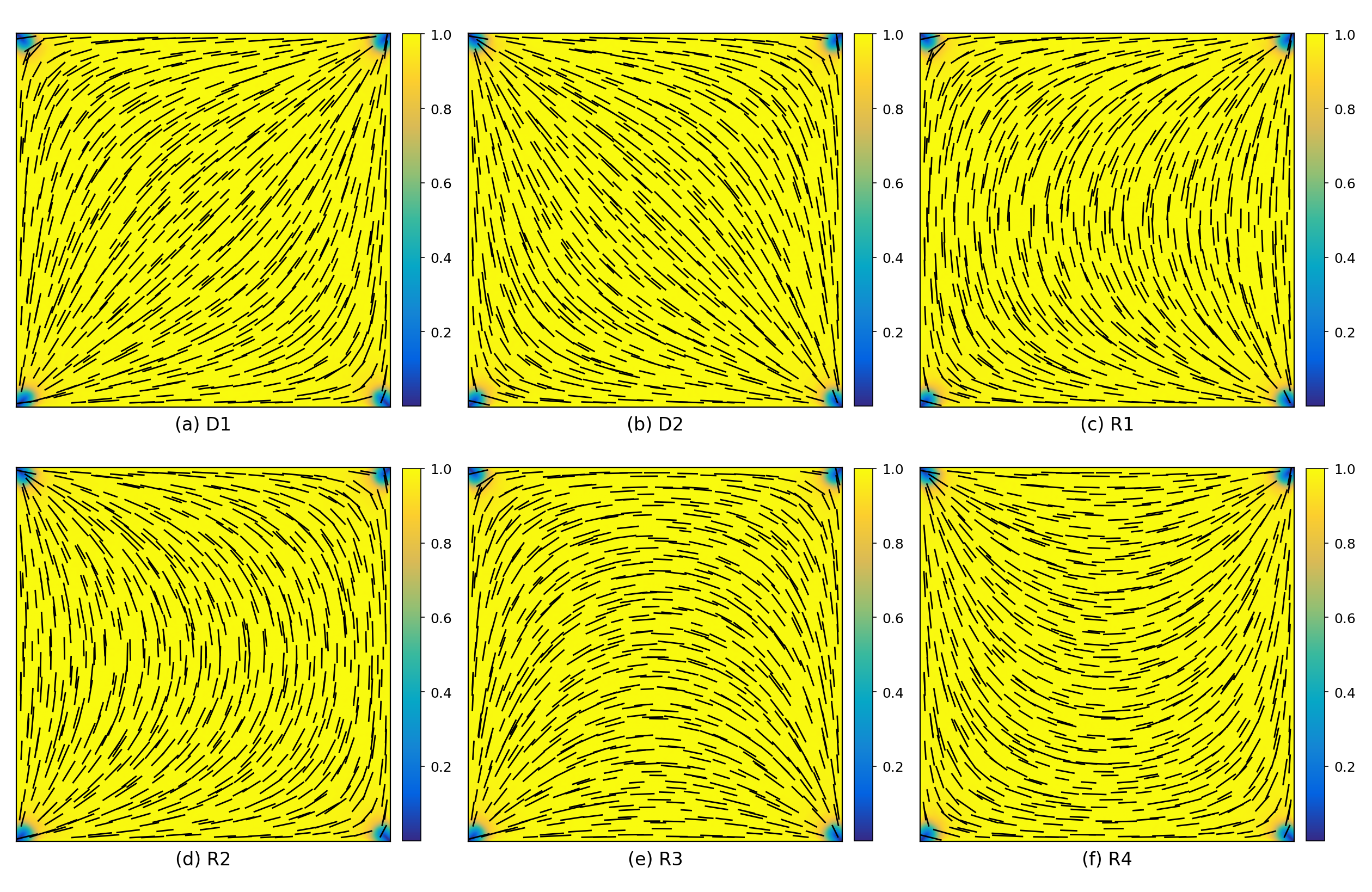}
\caption{\revtwo{The six solutions obtained by the Deflation--Deep-Ritz refinement (Table \ref{tab:relL2Differences}, Deflation--Deep-Ritz block): nematic director (white segments) over the scalar order parameter $\vert\Qvec\vert$ (color). Every branch exhibits $\vert\Qvec\vert \approx 1$ in the bulk, defect cores only at the four corners, no interior order-reconstruction walls, and the exact director topology of its stable state.}}
\label{fig:DDRSolutions}
\end{figure}
}

\revtwo{\begin{rem}
    We note that, in certain applications, the underlying energy functional may be
    nonsmooth and therefore not classically differentiable
    \cite{Baemmi19PGS,Bach23VEP}. This does not preclude the use of the proposed
    framework. When required by the optimization procedure, the original energy
    functional may be replaced by a suitable smooth approximation, for instance
    through Moreau--Yosida regularization or related infimal-convolution techniques \cite{More73PEAC,BauCom17CAMH, Ba26GMYR}. More generally, nonsmooth
    formulations may also be treated using tools from subdifferential and proximal
    calculus \cite{BauCom17CAMH,RockWets98VA}. We further emphasize that the underlying deflation principle is not intrinsically
    tied to neural-network hypothesis spaces. In principle, both the residual-based Deflation-PINN formulation and the variational Deflation-Deep-Ritz formulation can be combined with other suitable trial or hypothesis spaces, provided that the
    corresponding loss or energy functional can be evaluated and minimized over that
    space. Examples include reproducing kernel Hilbert spaces (RKHS)
    \cite{HofSchSmo08, krom26}, finite-dimensional random-feature spaces, such as random
    Fourier features \cite{RahimiRecht07}, or classical finite-element spaces. Thus, the deflation mechanism may
    be viewed more generally as a strategy for identifying multiple solutions of
    nonlinear variational or residual-minimization problems, rather than as a
    construction specific to neural-network approximations.
\end{rem}}

\revtwo{\subsection{A Second Benchmark: Allen--Cahn on the Unit Disk}
\label{subsec:AllenCahn}
To demonstrate that Deflation-PINNs are not tailored to the Landau--de Gennes problem, we apply the identical framework to a second multi-solution benchmark from a different equation family, on a different geometry, with homogeneous instead of non-homogeneous Dirichlet data: the Allen--Cahn-type problem
\[
    -\Delta u = \lambda u - u^3 \text{ in } B(0,1)\subset \mathbb R^2, \qquad u = 0 \text{ on } \partial B(0,1),
\] where we choose $\lambda=6$ in our experiments. 
The number of solutions of this problem is completely determined by the position of the coefficient $\lambda = 6$ relative to the Dirichlet spectrum of the Laplacian on the disk; this is the two-dimensional analogue of the classical Chafee--Infante structure \cite{ChafeeInfante1974}. Write $\lambda_1 = j_{0,1}^2 \approx 5.783$ and $\lambda_2 = j_{1,1}^2 \approx 14.68$ for the first two Dirichlet eigenvalues of the unit disk ($j_{\nu,1}$ denoting the first positive zero of the Bessel function $J_\nu$). For $\lambda \le \lambda_1$, testing the equation with $u$ gives $\Vert\nabla u\Vert_{L^2}^2 = \lambda \Vert u \Vert_{L^2}^2 - \Vert u\Vert_{L^4}^4$, which together with the Poincar\'e inequality $\lambda_1 \Vert u \Vert_{L^2}^2 \le \Vert\nabla u\Vert_{L^2}^2$ forces $u \equiv 0$. For $\lambda > \lambda_1$ a positive solution $u^*$ exists (the associated energy becomes negative along the principal eigenfunction, so its global minimizer is nontrivial); it is radially symmetric by the moving-plane theorem of Gidas--Ni--Nirenberg \cite{GidasNiNirenberg1979} and it is \emph{unique} among positive solutions by the Brezis--Oswald criterion \cite{BrezisOswald1986}, since the quotient $(\lambda u - u^3)/u = \lambda - u^2$ is strictly decreasing in $u > 0$. Finally, any sign-changing solution would have at least two nodal domains, and its restriction to each nodal domain $D$ is a one-signed solution of the same equation there, which forces $\lambda > \lambda_1(D)$; since two disjoint subdomains of a domain $\Omega$ always satisfy $\max_i \lambda_1(D_i) \ge \lambda_2(\Omega)$ by the variational characterization of eigenvalues, sign-changing solutions require $\lambda \ge \lambda_2 \approx 14.68$. Hence, for $\lambda_1 < 6 < \lambda_2$, the solution set is exactly $\{0, +u^*, -u^*\}$, i.e.\ the number of solutions the Deflation-PINN must recover is known rigorously to be $K = 3$. The same equation family is used in the random-initialization study of \cite{multipleSolPINN}.

We train a single Deflation-PINN with $K=3$ branches (shared trunk with $6$ hidden layers of width $100$, $\tanh$ activation, $p = 32$ branch features), enforcing the homogeneous Dirichlet condition exactly through the radial construction of Appendix \ref{app:radial} applied to the disk with $h(t) = 1 - t^2$, i.e.\ $u = (1-r^2)\cdot N_\theta$. The total loss is
\[
    \mathcal L_{total} = \alpha \cdot \sum_{k=1}^{3} \mathcal E_G(u_k) + \beta \cdot \mathcal L_{Def}, \qquad \alpha = 100, \quad \beta = 1,
\]
where $\mathcal E_G$ is the mean squared PDE residual of branch $k$ and $\mathcal L_{Def}$ is the deflation loss of Section \ref{Subsection Deflation-PINNs} with $d_{min} = 0.2$, the pairwise distance being measured in the discrete mean-absolute metric of the implementation (see the Reproducibility Statement). Training uses AdamW for $40{,}000$ epochs followed by an L-BFGS polishing stage in double precision (full details in the Reproducibility Statement); on this smooth benchmark, direct residual training attains reference accuracy without a classical refinement stage.

The choice of $d_{min}$ deserves emphasis, since this benchmark exhibits the failure mode discussed in our Future Work section in a measurable way: the separation between the trivial branch and $\pm u^*$ is $\Vert u^* - 0\Vert \approx 0.272$ (in the discrete mean-absolute metric used by the implementation). An initial experiment with $d_{min} = 0.4 > 0.272$ therefore \emph{cannot} terminate its deflation term---the hinge penalty keeps pushing the branches apart beyond the true separation---and it converged to visibly inflated fields (root-mean-square value $\approx 0.47$ instead of $0.32$ for $u^*$, mean residual $\approx 0.14$). Choosing $d_{min}=0.2$, below the true minimal separation, removes the obstruction. The practical rule is that $d_{min}$ must lower-bound the minimal pairwise distance of the true solutions; when it does not, the deflation loss itself reports the problem, since it cannot reach zero.

With $d_{min}=0.2$ the Deflation-PINN recovers all three solutions from a single training run, see Figure \ref{fig:ACSolutions}.

\paragraph{Reference solution.} The comparison uses no discretized PDE solve and no external data: by the symmetry and uniqueness results quoted above, the positive solution is radial and unique, so $u^*$ is \emph{exactly} characterized by the two-point boundary value problem $u'' + u'/r + 6u - u^3 = 0$, $u'(0)=0$, $u(1)=0$, which we solve by shooting. For an amplitude $a>0$ we integrate the initial value problem from $r_0 = 10^{-8}$ with $u(r_0)=a$, $u'(r_0)=0$ (adaptive Runge--Kutta of order 4(5), relative tolerance $10^{-11}$, absolute tolerance $10^{-12}$) and solve $u(1;a)=0$ for $a$ by bisection to a tolerance of $10^{-13}$. A sign change is guaranteed: for small $a$ the linearization gives $u(1;a) \approx a\,J_0(\sqrt 6) < 0$ because $\sqrt 6 > j_{0,1}$, while for large $a$ the cubic term dominates and $u(1;a) > 0$. This yields $u^*(0) = 0.6170166479$ and $\Vert u^*\Vert_{L^2(B)} = 0.571449$; the trivial solution and $-u^*$ complete the solution set. The reference is therefore accurate to the ODE integration tolerance, several orders below the errors we report.

Errors are evaluated on a polar grid of $401\times401$ points that is disjoint from the training collocation set, with the $L^2(B)$ integrals computed by tensorized Simpson quadrature including the area element $r\,dr\,d\vartheta$. The results are collected in Table \ref{tab:ACerrors}: the relative $L^2$ errors of the two nontrivial branches are of order $10^{-5}$--$10^{-4}$, and the branch matched to the trivial solution vanishes to $2\cdot 10^{-5}$. This demonstrates that the deflation mechanism is problem-independent and that, on a benchmark without singular structures, Deflation-PINNs attain the accuracy customary for PINN methods.

\begin{table}[H]
\begin{center}
\begin{tabular}{ | m{8em} | m{7em} | m{7em} | m{8em} |}
    \hline
    Branch matched to & Rel.\ $L^2$ error & Abs.\ $L^2$ error & Mean $\vert$residual$\vert$ \\
    \hline
    $+u^*$ & $4.7\cdot 10^{-5}$ & $2.7\cdot 10^{-5}$ & $5.5\cdot 10^{-4}$ \\
    \hline
    $-u^*$ & $1.4\cdot 10^{-4}$ & $8.2\cdot 10^{-5}$ & $6.1\cdot 10^{-4}$ \\
    \hline
    $u \equiv 0$ & --- & $2.0\cdot 10^{-5}$ & $1.3\cdot 10^{-4}$ \\
    \hline
\end{tabular}
\end{center}
\caption{\revtwo{Allen--Cahn benchmark: errors of the three branches learned by a single Deflation-PINN against the radial shooting reference $u^*$ ($\Vert u^*\Vert_{L^2(B)} = 0.571449$), integrated on an independent $401\times401$ polar grid with area weights. For the trivial branch the relative error is not defined, so we report the absolute $L^2$ norm; the PDE residual is the mean of $\vert \Delta u + 6u - u^3\vert$ over a fine evaluation grid.}}
\label{tab:ACerrors}
\end{table}

\begin{figure}[H]
\centering
\includegraphics[width=\textwidth]{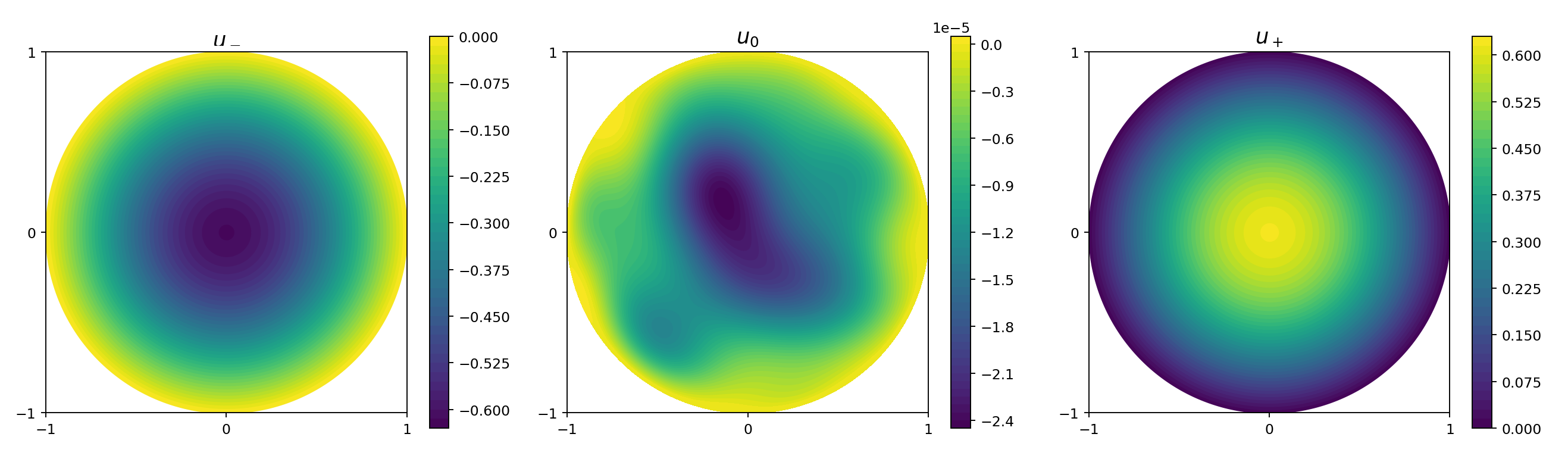}
\caption{\revtwo{The three solutions of the Allen--Cahn benchmark on the unit disk learned simultaneously by a single Deflation-PINN with $K=3$ branches, evaluated on a fine polar grid: (left) $u_-$, the branch matched to $-u^*$; (middle) $u_0$, the branch matched to the trivial solution (note the $10^{-5}$ color scale); (right) $u_+$, the branch matched to $+u^*$, with maximum $\approx 0.617$ in agreement with the shooting reference $u^*(0)\approx 0.61702$. Quantitative errors are reported in Table \ref{tab:ACerrors}.}}
\label{fig:ACSolutions}
\end{figure}
}

\revtwo{
\section{Discussion}
\label{sec:Discussion}

We introduced Deflation-PINNs, a framework for learning multiple solution branches of nonlinear PDEs within a single neural-network model. The method combines a finite-branch parameterization inspired by DeepONets with a physics-informed residual loss and a pairwise deflation term. The shared trunk provides a common representation of the spatial dependence, while learnable branch vectors encode the individual candidate solutions. The approximation results developed in Section \ref{Subsection Deflation-PINNs} and for the hard-constrained architecture establish that this parameterization is sufficiently expressive to approximate a prescribed finite family of solutions. The numerical experiments complement this approximation theory by clarifying both the capabilities and the limitations of the proposed training procedure.

For the Landau--de Gennes benchmark, a single Deflation-PINN run with \(K=6\) produced six branches that, under the gradient-flow/Newton census of Section \ref{subsec:refinement}, belong to the basins of the six distinct stable equilibrium states. The discovery-stage approximations are intentionally coarse: their relative \(L^2\) errors range from approximately \(0.30\) to \(0.59\). Nevertheless, they contain enough information to identify the corresponding equilibrium basins. Starting from these branches, the purely neural Deflation--Deep-Ritz refinement of Section \ref{subsec:DDR} reduces the relative \(L^2\) errors to \(1.2\%\)--\(1.3\%\), while the classical gradient-flow/Newton refinement reaches the mesh-converged reference solutions to numerical precision. In this sense, the main value of the residual-based Deflation-PINN on this benchmark is as a \emph{multi-solution discovery mechanism}: it produces several basin-correct initializations simultaneously, which may subsequently be refined to high accuracy.

At the same time, the additional experiments show that pairwise deflation should not be interpreted as a guarantee of distinct basin membership. The deflation loss constrains distances between the learned fields, whereas two fields separated by more than \(d_{\min}\) may still lie in the basin of the same equilibrium. This distinction is visible in the restart study of Appendix \ref{app:convergence}: although the particular run reported in Table \ref{tab:relL2Differences} recovers all six stable LdG states, independent runs at the same configuration recover between two and four distinct basins. The cumulative union of certified states nevertheless reaches all six equilibria after five runs. Thus, the practically relevant procedure is not to assume completeness from a single training run, but to combine simultaneous multi-branch discovery with a problem-dependent certification step and, when necessary, a small number of independent restarts. For the LdG problem, the gradient-flow/Newton census provides such a certification without access to reference solutions.

The experiments also reveal a useful division of labor between residual- and energy-based objectives. Residual minimization is insensitive to stability and is therefore suitable for exploring different parts of the solution landscape, but on the LdG benchmark its loss geometry permits substantial bulk amplitude errors. Energy minimization removes this failure mode and provides an effective refinement mechanism, but direct Deflation--Deep-Ritz training from random initialization preferentially populates the lowest-energy basins and does not provide comparable basin coverage. The two stages are therefore complementary: residual-based deflation is used for exploration, whereas energy-based optimization is used for refinement.

The Allen--Cahn benchmark provides a complementary test in which the complete solution set is known independently. For the chosen parameter value, the problem has exactly three solutions, and a single Deflation-PINN with \(K=3\) recovers all three with errors of order \(10^{-5}\)--\(10^{-4}\) without subsequent classical refinement. This experiment also exposes the role of \(d_{\min}\) particularly clearly: choosing \(d_{\min}\) above the true minimal inter-solution separation makes the deflation constraint incompatible with the solution set and prevents convergence to the correct branches. Together with the LdG experiments, this demonstrates that Deflation-PINNs are applicable beyond a single PDE or geometry, while also showing that the separation scale and basin coverage are intrinsic parts of the multi-solution optimization problem.

\subsection*{Future Work}

A first direction is to remove the requirement that the number of branches \(K\) be prescribed in advance. A natural extension is an adaptive discovery procedure in which training begins with a small number of branches and additional branches are introduced only as needed. Crucially, the present experiments show that a small residual together with a vanishing deflation penalty is not, by itself, sufficient evidence that a newly added branch represents a previously unseen solution. An adaptive method should therefore combine branch generation with a certification or identification step. For the LdG problem, the gradient-flow/Newton census used in Section \ref{subsec:refinement} provides such a mechanism: a candidate is accepted only when it converges to an equilibrium not already represented in the current collection. For more general PDEs, however, an analogous certification procedure may need to be based on continuation, Newton-type solvers, stability information, conserved quantities, or other problem-specific diagnostics. Developing certification criteria that extend beyond variational problems is therefore an important part of making adaptive-\(K\) Deflation-PINNs broadly applicable.

A second open problem is the choice of the separation mechanism itself. The Allen--Cahn experiment shows that \(d_{\min}\) cannot exceed the minimal separation of the solutions one wishes to recover, while the LdG experiments show that choosing a smaller admissible value does not prevent multiple branches from occupying the same basin. Moreover, the adaptive-\(d_{\min}\) and scale-free repulsion experiments of Appendix \ref{app:convergence} indicate that modifying the pairwise separation term alone does not resolve basin duplication. Future work should therefore investigate separation criteria that use more information than instantaneous field-space distance, for example continuation-based branch tracking, physically meaningful observables, or diversity measures coupled to a certification procedure. An adaptive estimate of the relevant distance scale may still be useful, but the present results suggest that it should complement rather than replace basin identification.

A third direction concerns unstable solutions. The residual formulation is stability-blind, and Section \ref{subsec:refinement} shows that the LdG problem contains genuine unstable wall-type solutions in addition to its six stable equilibria. This suggests using residual-based deflation with more branches than the number of stable states to explore the saddle-point structure of the solution landscape. Such a study would require not only discovering additional critical points but also classifying them, for example through linearized stability information or an appropriate notion of Morse index. This would extend Deflation-PINNs from the discovery of multiple stable equilibria to the computation of richer nonlinear solution landscapes.

Finally, both computational efficiency and theoretical guarantees remain open. On the LdG benchmark, the accuracy of the purely neural refinement is ultimately limited by the resolution of the defect cores and by the numerical quadrature used for the energy, while classical post-processing reaches reference accuracy at comparatively low additional cost. Adaptive quadrature, localized sampling near singular structures, and architectures better matched to multiple spatial scales are natural ways to reduce this gap. At the theoretical level, the present approximation results establish expressive power but do not provide guarantees that optimization will discover every solution or every basin. Establishing conditions under which deflated training yields reliable basin coverage---and quantifying how this depends on \(K\), \(d_{\min}\), initialization, and the underlying solution landscape---would provide a substantially stronger foundation for the method.}

\section*{Acknowledgments}
A.Bacho acknowledges support from the Air Force Office of Scientific Research (AFOSR) under awards FA9550-20-1-0358 (MURI: “Machine Learning and Physics-Based Modeling and Simulation”) and FA9550-24-1-0237, the U.S. Department of Energy (DOE), Office of Science, Office of Advanced Scientific Computing Research (ASCR) under award DE-SC0023163 (SEA-CROGS: “Scalable, Efficient and Accelerated Causal Reasoning Operators, Graphs and Spikes for Earth and Embedded Systems”), and the Office of Naval Research (ONR) under award N00014-25-1-2035. R. M. acknowledges the support of
Department of Science and Technology, India (INSPIRE Faculty Fellowship, IFA24-MA
208).

\bibliographystyle{abbrvnat}
\bibliography{bib_aras}
\newpage

\appendix
\section*{Appendix}

\section{Reproducibility Statement}
\revtwo{The experiments were performed on an NVIDIA V100 or an RTX 4060-Ti (8GB) in single precision, with the double-precision stages as declared in the paragraphs below, and the pipeline reproduces across these machines. The code for all experiments, including the revision harness, launch scripts and analysis, is available at \url{https://github.com/SeanDisaro/DeflationPINNs}}.
We evaluated the PIML loss at collocation points on a $33 \times 33$-grid and took the Riemann sum for this grid. The grid covers the domain $[0+\delta,1-\delta ]^2$ with $\delta=0.001 $. The safety distance is to facilitate training, and we can do this safely, as the boundary condition is enforced.\\
We used $d_{min} = 0.4$. \revtwo{This value, originally found empirically, is now quantified: the measured pairwise separations of the true solutions bracket the admissible range (Section \ref{subsec:refinement}), and the automated bisection of Appendix \ref{app:convergence} certifies the choice ($d_{min} = 0.4 \approx 0.8\, d^*$, with $d^*$ the trainability edge). We note that increasing $d_{min}$ does not remove basin duplication (Section \ref{subsec:refinement}); completeness is handled by the restart-with-certification protocol, not by tuning $d_{min}$.}
Thus, the deflation loss we used is given by
\[
    \mathcal L_{Def} (\mathbf G) := {1 \over 15} \sum _ {i = 1}^6\sum _ {j = i+ 1} ^6 \max \Big( 1 - {5 \over 2}\revtwo{|| \mathbf G(i, \cdot ) -  \mathbf G(j, \cdot )||_{L^2 ([0,1]^2, \mathbb R^2)}} , 0 \Big).
\]
To compute the \revtwo{$L^2([0,1]^2, \mathbb R^2)$} norm for the deflation loss \revtwo{(as the root-mean-square over the collocation points, which approximates the $L^2$ norm on the area-one domain)}, we use the same grid / collocation points, as for the PIML loss. 
The total loss function is of the form
\[
    \mathcal L _ {total } (\mathbf G) := \alpha\cdot \sum \limits _{i =1}^6  \mathcal E_G (\mathbf G(i, \cdot)) + \beta\cdot \mathcal L_{Def} (\mathbf  G)
\]
and we used \revtwo{$\alpha = 0.01$ and $\beta = 100$ (these values, as well as the learning rate below, correct two typos in the previous version of this statement and match the released code).}
In our experiments, we use the Adam optimizer to train for $10,000$ epochs with an initial learning rate of \revtwo{$1 \cdot 10 ^{-4}$}.
$\tau: \mathbb R ^2 \to \mathbb R ^p$ is of width $4000$ and has one layer. We use tanh as the activation function. We set $p = 16$, i.e.\ $\tau$ has $2p = 32$ outputs and each branch vector $\beta^k \in \mathbb R^{32}$ (Section \ref{subsec:Solving LdG with DeflationPINNs}).

\revtwo{\paragraph{Reference solutions and refinement (Section \ref{subsec:refinement}).} The reference solver uses the standard 5-point finite-difference Laplacian on uniform grids ($h = 1/256$ and $1/512$) with the exact trapezoidal Dirichlet data, and a damped Newton iteration (step halving on residual increase) on the $2\times$-coupled linearization, solved by a sparse direct factorization; termination once the residual falls below $10^{-10}$ (in practice all states terminate at $5\cdot 10^{-12}$ after three steps, cf.\ Section \ref{subsec:refinement}). The refinement flow is the semi-implicit scheme $(I - \Delta t\, \epsilon^2 \Delta_h)\Qvec^{n+1} = \Qvec^n + \Delta t\, 2(1-\vert \Qvec^n\vert^2)\Qvec^n$ with $\Delta t = 0.1$ for $3{,}000$ steps at $h = 1/256$ (one sparse factorization, reused), followed by the damped Newton iteration; each branch requires under two minutes on CPU.}

\revtwo{\paragraph{Deflation--Deep-Ritz refinement (Section \ref{subsec:DDR}).} Each branch is represented by an independent network: $128$ random Fourier features $x \mapsto (\sin, \cos)(2\pi B^\top x)$ with $B \sim \mathcal N(0, 16^2)$ (wavelengths matched to the core scale), followed by a tanh MLP with $3$ hidden layers of width $128$; component one carries the harmonic boundary extension as hard constraint, component two the boundary-vanishing factor alone. Self-distillation: $4{,}000$ Adam epochs fitting the discovery-stage fields. Ritz stage: composite-midpoint quadrature on cell-centered grids ($192^2$ for Adam, $16{,}000$ epochs, initial learning rate $10^{-3}$ with step decay; then $384^2$ for L-BFGS in double precision, strong Wolfe line search), loss weights $\alpha = 0.01$, $\beta = 100$, $d_{min} = 0.8$ in the full vector-field metric; finally a quadrature continuation replacing the midpoint rule by tensorized Simpson weights on a $321^2$ nodal grid ($200$ further double-precision L-BFGS steps, warm-started from the midpoint solution). The first two stages take a few GPU hours on a single V100; the Simpson continuation was executed in double precision on a 32-core CPU ($\approx 30$ hours).}

\revtwo{\paragraph{Allen--Cahn benchmark (Section \ref{subsec:AllenCahn}).} The trunk network has $6$ hidden layers of width $100$ with tanh activation and $p=32$; $K = 3$ branches. Collocation points are the nodes of a $33\times 33$ grid over $[-1,1]^2$ restricted to the open unit disk ($\approx 800$ points). The PIML term is the mean squared PDE residual summed over branches with weight $\alpha = 100$; the deflation term uses the discrete mean-absolute distance with $d_{min}=0.2$, hinge weight $\beta = 1$. Training: AdamW with initial learning rate $10^{-4}$ and a step decay (factor $0.8$ every $2{,}000$ epochs) for $40{,}000$ epochs in single precision, followed by $100$ L-BFGS steps (strong Wolfe line search, history $50$) in double precision. These runs were performed on a single V100 GPU; total training time is below one hour.}

\revtwo{\paragraph{Convergence, restart, variant and bisection studies (Appendix \ref{app:convergence}).} All appendix runs use the discovery protocol above with one quantity varied at a time; the restart study varies only the random initialization. Certification of basin membership uses the gradient-flow/Newton procedure of Section \ref{subsec:refinement} (CPU, minutes per branch). The $d_{min}$ bisection uses reduced budgets per query (Allen--Cahn: $15{,}000$ Adam epochs and $30$ L-BFGS steps; LdG: $10{,}000$ Adam epochs) with the residual-based classifier specified in Appendix \ref{app:convergence}; queries take minutes (Allen--Cahn) to $\approx 25$ minutes (LdG) each. These runs were executed in single precision on a V100 or an RTX 4060-Ti.}

\revtwo{\section{Radial Extensions on Star-Shaped Domains}
\label{app:radial}
This appendix contains the general construction of the boundary-data extension used for the hard constraints; in the main text only its instances on the square and on the disk are used.

Let $\Omega \subset \mathbb R^d$ be a bounded, open, and star-shaped domain (see Definition \ref{defStarDomain}) with respect to $x_0 \in \Omega$. We consider $\Omega$ in spherical coordinates with center $x_0$, i.e., for a point $x \in \Omega$ given in Cartesian coordinates we write
\begin{align*}
    x \cong (\varphi_{x}^{x_0}, \theta^{x_0}_x , r^{x_0}_x) &:= (\varphi_{x-x_0}^{0}, \theta^{0}_{x-x_0} , r^{0}_{x-x_0})\\
    &:= (\varphi_{x-x_0}, \theta_{x-x_0} , \Vert x- x_0\Vert _2) \in [0, 2\pi ) \times [0, \pi)^{d-2} \times [0, \infty ),
\end{align*}
the usual representation of $x-x_0$ in spherical coordinates.
Each boundary point $x \in \partial \Omega$ corresponds to one and only one angle, i.e. $\partial \Omega \longleftrightarrow [0, 2\pi ) \times [0, \pi)^{d-2} \longleftrightarrow \mathbb S^1 \times \mathbb S^{d-2}$, where $\mathbb S^n =  \partial B(0,1) \subset \mathbb R ^{n+1}$, and these mappings are continuous. This gives rise to the following definition.
\begin{Def}
    \label{def:radial boundary distanc(rbd function)}
    Let $\Omega \subset \mathbb R^d$ be a bounded, open, and star-shaped domain with respect to $x_0 \in \Omega$. We define the radial boundary distance function (RBD function) $r_{\Omega , x_0}: \bar \Omega \setminus x_0 \to [0, \infty) $ of $\Omega$ with respect to $x_0$, such that for all $x \in \bar \Omega \setminus \{x_0 \} $ we have
    \[
        r_{\Omega , x_0}(x) :=  \Vert x_b -x_0\Vert_2,
    \]
    where $x_b \in \partial \Omega$ is such that $x \in \{ (1-t)x_0 + t x_b ~|~ t \in [0,1]\}$ lies on the line segment connecting $x_b$ and $x_0$.
\end{Def}
Definition \ref{def:radial boundary distanc(rbd function)} is to be understood so that every point (except the center) is mapped radially to a boundary point, and the function returns the distance from the center to that boundary point; in particular $(\varphi^{x_0}_x, \theta^{x_0}_x) = (\varphi^{x_0}_{x_b}, \theta^{x_0}_{x_b})$.
\begin{example}
    The ball $B(x_0, R) \subset \mathbb R^d$ with radius $R>0$ and center $x_0$ is star-shaped with respect to $x_0$, and its RBD function is $r_{B(x_0, R) , x_0}:x \mapsto R$ for all $x \in \bar B(x_0, R) \setminus \{x_0 \}$.
\end{example}
Spherical continuations are now defined in terms of the RBD function.
\begin{Def}
    \label{def radial Extrapolation}
    Let $\Omega \subset \mathbb R^d$ be a bounded, open and star-shaped domain with respect to $x_0$ and let $r_{\Omega , x_0}$ be a RBD function for $\Omega$. Let $h:[0,1] \to [0,1]$ be continuous with $h\vert _{[0,1)}> 0 $, $h(0) = 1$ and $h(1) = 0$, and let $f \in C(\partial \Omega )$. Define
    \[
        \tilde\omega (x) : \begin{cases*}
            \Omega \to [0,1] \\
            x \mapsto\begin{cases*}
                h({\Vert x- x_0\Vert \over r_{\Omega , x_0} (x)}),\text{ if }x \neq x_0\\
                1,\text{ if }x = x_0.
            \end{cases*}
        \end{cases*} \text{ and }       \tilde f : \begin{cases*}
            \Omega \to \mathbb R \\
            x \mapsto f(\varphi _{x}^{x_0},\theta _{x}^{x_0}) \cdot\left(1-\tilde\omega (x) \right).
        \end{cases*}
    \]
    Then $\tilde\omega, \tilde f \in C( \bar \Omega)$, $\tilde \omega = 0 $ on $\partial \Omega$, $ \tilde \omega > 0$ in $\text{int}({\Omega})$, and $\tilde f \vert _{\partial \Omega} = f$.
\end{Def}
If $h$ is smooth and $\Omega$ and $f:\partial \Omega \to \mathbb R$ are smooth, then $\tilde \omega$ and $\tilde f$ are smooth. The general difficulty of this approach is to find an RBD function for the domain, which in general is not easy. Note that $\tilde \omega$ satisfies the conditions required of $\omega$ in the ansatz of Section \ref{Section Hard Constraints For Boundary Conditions}; on the square, however, we use only $\tilde f$ (for the extension of the boundary data) and take the $C^\infty$-construction of \cite{lu2021deepxde} for $\omega$, whereas on the disk (Section \ref{subsec:AllenCahn}) the boundary data vanish and $\tilde\omega$ itself, with $h(t) = 1 - t^2$, serves as the boundary-vanishing factor $\omega$.}

\revtwo{\section{Convergence and Deflation-Variant Studies}
\label{app:convergence}
This appendix collects the convergence studies in collocation density and network width and first empirical tests of the deflation-loss variants of Section \ref{Subsection Deflation-PINNs}. All runs use the discovery protocol and hyperparameters of Section \ref{sec:Experiments} (full vector-field deflation, $d_{min}=0.4$, $10{,}000$ Adam epochs, fixed seed except in the restart study below), varying one quantity at a time; they were executed in single precision on a single GPU (NVIDIA V100 or RTX~4060~Ti; the pipeline reproduces across the two machines). For each run, every branch is matched to a reference state, and---as in Section \ref{subsec:refinement}---flowed under the LdG gradient flow followed by Newton iteration; ``distinct basins'' counts how many \emph{different} stable states the six branches reach under this census.

\begin{table}[H]
\begin{center}
\begin{tabular}{ | m{13em} | m{5.5em} | m{13em} |}
    \hline
    Configuration & Distinct basins & Rel.\ $L^2$ error (min\,/\,median\,/\,max) \\
    \hline
    \multicolumn{3}{|l|}{\emph{Collocation sweep (width $4000$):}} \\
    \hline
    $17\times 17$ & $4/6$ & $0.29$ / $0.59$ / $0.81$ \\
    \hline
    $25\times 25$ & $4/6$ & $0.30$ / $0.47$ / $0.92$ \\
    \hline
    $33\times 33$ (Section \ref{sec:Experiments}) & $6/6$ & $0.30$ / $0.31$ / $0.59$ \\
    \hline
    $49\times 49$ & $5/6$ & $0.29$ / $0.29$ / $0.87$ \\
    \hline
    \multicolumn{3}{|l|}{\emph{Width sweep ($33\times 33$ collocation):}} \\
    \hline
    width $500$ & $3/6$ & $0.30$ / $0.61$ / $0.98$ \\
    \hline
    width $1000$ & $6/6$ & $0.30$ / $0.30$ / $0.59$ \\
    \hline
    width $2000$ & $5/6$ & $0.30$ / $0.44$ / $0.86$ \\
    \hline
    width $4000$ (Section \ref{sec:Experiments}) & $6/6$ & $0.30$ / $0.31$ / $0.59$ \\
    \hline
\end{tabular}
\end{center}
\caption{\revtwo{Convergence of the discovery stage in collocation density and network width. Per configuration: number of distinct stable states reached by the six branches under the gradient-flow/Newton census, and the minimum, median and maximum relative $L^2$ error of the branches against the mesh-converged reference (branches matched to states by minimum-cost assignment).}}
\label{tab:convergence}
\end{table}

Two features of Table \ref{tab:convergence} stand out. First, the attainable accuracy is \emph{flat} across both sweeps: the best-matched branches sit at relative errors $0.29$--$0.30$ in every configuration, independent of the collocation density ($17^2$--$49^2$) and of the width ($500$--$4000$). The discovery-stage error does therefore not behave like a discretization error, consistent with the mechanism identified in Section \ref{subsec:refinement}: it is dominated by the bulk amplitude deficit that the flat reaction-residual landscape tolerates, and neither denser collocation nor larger capacity removes it---the refinement stages of Sections \ref{subsec:refinement} and \ref{subsec:DDR} do. Second, what does vary between configurations is \emph{basin coverage} (between $3/6$ and $6/6$, without a monotone trend): a mis-configured run does not fail by losing accuracy but by duplicating basins, which is precisely the failure mode discussed in Section \ref{subsec:refinement} (deflation constrains distances, not basin membership).

\paragraph{Feature-space deflation.} Our architecture makes the feature-space variant of the deflation loss particularly easy to test: each solution is encoded by its branch vector $\beta^k$, so we replace the field-space deflation distance by the normalized Euclidean distance between branch vectors, with hinge, $d_{min}$ and training budget unchanged. The run separates the branch vectors as intended, but the fields are not controlled by this separation: the six branches occupy only three distinct basins (D1, D2, R1), and no branch reaches the field-space accuracy floor (relative errors $0.59$--$0.66$, against $0.30$--$0.59$ for field-space deflation under the identical budget). The mechanism is transparent---well-separated parameter vectors can render nearly identical fields---so separation should be enforced in the space in which distinctness is actually required.

\paragraph{Single-stage energy training and adaptive separation terms.} We further tested whether the discovery stage can be skipped altogether by minimizing the Deflation--Deep-Ritz loss of Section \ref{subsec:DDR} directly from random initialization, in three variants: (i) fixed $d_{min}=0.4$; (ii) $d_{min}$ annealed from $0.4$ to $0.9$ during training (an adaptive deflation parameter); (iii) the scale-free inverse-power repulsion of Section \ref{Subsection Deflation-PINNs} with exponent $q=2$ and repulsion scale $0.5$ in place of the hinge. All three collapse in the same way: the six branches distribute over only the \emph{two lowest-energy} states D1 and D2, with the minimum pairwise branch distance settling exactly at the imposed constraint ($0.400$ and $0.900$ for the hinge variants (i) and (ii), and $0.535$, just above the repulsion scale $0.5$, for (iii)). The separation term is satisfied---by pairing up within the two dominant basins rather than by spreading over all six. Energy descent selects basins by depth and volume, and no pairwise separation term we tested, fixed, annealed, or scale-free, redirects branches into the four remaining basins. This is the clearest evidence for the division of labor in our pipeline: residual-based deflation \emph{discovers} basins, energy-based training \emph{refines} within them. Consistently, over-generation within the residual protocol does not enforce full coverage either: with $K=11$ branches at $d_{min}=0.4$, nine branches certify under the census of Section \ref{subsec:refinement} yet occupy only four distinct stable states.

\paragraph{Restart study at the published configuration.} Table \ref{tab:restarts} reports eight independent discovery runs at the published configuration ($33\times 33$ collocation, width $4000$)---identical protocol, different random initializations---together with the basins certified by the gradient-flow/Newton census and the cumulative union. Individual runs certify between two and four basins (mean $3.1$); the certified union reaches $6/6$ after five runs and stays complete. The per-state frequencies (D2 in $8/8$ runs, D1 and R3 in $6/8$, R4 in $3/8$, R1 and R2 in $1/8$ each) quantify the basin anisotropy that a single run faces and confirm that duplication, not accuracy, is the practical failure mode of the discovery stage; the restart-with-certification protocol of Section \ref{subsec:refinement} is the systematic remedy.

\begin{table}[H]
\begin{center}
\begin{tabular}{ | m{3em} | m{13em} | m{5em} | m{7em} |}
    \hline
    Run & Certified basins & Distinct & Cumulative union \\
    \hline
    1 & D1, D2, R3 & $3/6$ & $3/6$ \\
    \hline
    2 & D1, D2, R3, R4 & $4/6$ & $4/6$ \\
    \hline
    3 & D1, D2, R3 & $3/6$ & $4/6$ \\
    \hline
    4 & D1, D2, R3, R4 & $4/6$ & $4/6$ \\
    \hline
    5 & D2, R1, R2, R3 & $4/6$ & $6/6$ \\
    \hline
    6 & D1, D2 & $2/6$ & $6/6$ \\
    \hline
    7 & D2, R3 & $2/6$ & $6/6$ \\
    \hline
    8 & D1, D2, R4 & $3/6$ & $6/6$ \\
    \hline
\end{tabular}
\end{center}
\caption{\revtwo{Restart study: eight independent discovery runs at the published configuration, the stable states certified by the gradient-flow/Newton census, and the cumulative union of certified basins. The union covers all six states after five runs; the census is unsupervised, so ``stop when consecutive restarts add no new basin'' is a practical stopping rule.}}
\label{tab:restarts}
\end{table}

\paragraph{Automated selection of $d_{min}$ by bisection.} The admissible range of Section \ref{subsec:refinement} can be located automatically, without reference solutions. The final PDE residual of a (reduced-budget) discovery run is a two-regime function of $d_{min}$---small below the trainability edge, elevated above it---so a bisection that classifies each query by ``final residual below the geometric mean of the current endpoint residuals'' localizes the edge in logarithmically many runs. On the Allen--Cahn benchmark (endpoints $0.2/0.4$; $15{,}000$ Adam epochs and $30$ L-BFGS steps per query, minutes each) four queries yield the interval $[0.275, 0.288]$, which lies within $6\%$ of the true minimal separation $0.272$ of Section \ref{subsec:AllenCahn} (the query at $0.275$, one percent above the separation, is still classified as trainable because the branches absorb the marginal excess by a slight amplitude inflation, visible as a $40\%$ residual increase in Table \ref{tab:bisect}); the mechanism is visible in the hinge itself, whose terminal value is zero for the queries at or below $0.275$ and positive for those at or above $0.2875$, while the residual jumps by a factor $3$--$5$ across the same threshold. The bisection thus recovers the geometry of the solution set from training behavior alone. On the LdG problem (endpoints $0.4/0.85$; $10{,}000$ epochs per query) the same procedure places the trainability edge in $[0.48, 0.51]$---far below the geometric bound $1.086$, with the hinge satisfied on both sides of the edge: here the transition is set by the early-training dominance of the deflation term discussed in Section \ref{subsec:refinement}, and the bisection quantifies the safety margin of the published value ($d_{min} = 0.4 \approx 0.8\, d^*$). Repeat runs with a second seed refine this picture: $d_{min}=0.48$ is reliably trainable (final residuals $2.1$ for both seeds), while at $0.51$ the outcome is seed-dependent (residuals $5$--$33$)---the LdG edge is a narrow transition band around $0.5$ rather than a sharp point, as expected for an optimization-driven (rather than geometric) threshold. The resulting recipe is: bisect for the edge $d^*$, choose $d_{min}$ a safety factor below it, and leave completeness to the restart-with-certification protocol above. Table \ref{tab:bisect} lists every query of both bisections, so the two transition signatures---the hinge signature on Allen--Cahn, the residual plateau on LdG---can be read off directly.}

\begin{table}[H]
\begin{center}
\begin{tabular}{ | m{8em} | m{7em} | m{6.5em} | m{5em} |}
    \hline
    $d_{min}$ & Final residual & Terminal hinge & Verdict \\
    \hline
    \multicolumn{4}{|l|}{\emph{Allen--Cahn (true minimal separation $\approx 0.272$):}} \\
    \hline
    $0.2$ (endpoint) & $4.0\cdot 10^{-5}$ & $0.000$ & good \\
    \hline
    $0.25$ & $4.2\cdot 10^{-5}$ & $0.000$ & good \\
    \hline
    $0.275$ & $5.9\cdot 10^{-5}$ & $0.000$ & good \\
    \hline
    $0.2875$ & $1.8\cdot 10^{-4}$ & $0.0028$ & bad \\
    \hline
    $0.3$ & $2.0\cdot 10^{-4}$ & $0.027$ & bad \\
    \hline
    $0.4$ (endpoint) & $1.4\cdot 10^{-4}$ & $0.193$ & bad \\
    \hline
    \multicolumn{4}{|l|}{\emph{LdG (geometric bound $1.086$; edge is optimization-driven):}} \\
    \hline
    $0.4$ (endpoint) & $2.15$ & $0.000$ & good \\
    \hline
    $0.4562$ & $5.25$ & $0.000$ & good \\
    \hline
    $0.4843$ & $2.12$ & $0.000$ & good \\
    \hline
    $0.4843$ (seed 1) & $2.09$ & $0.000$ & good \\
    \hline
    $0.5125$ & $33.1$ & $0.000$ & bad \\
    \hline
    $0.5125$ (seed 1) & $5.16$ & $0.000$ & good$^{\dagger}$ \\
    \hline
    $0.625$ & $110$ & $0.000$ & bad \\
    \hline
    $0.85$ (endpoint) & $17.6$ & $0.000$ & bad \\
    \hline
\end{tabular}
\end{center}
\caption{\revtwo{All queries of the two $d_{min}$ bisections, sorted by $d_{min}$. Allen--Cahn block: the terminal hinge value is zero for $d_{min} \le 0.275$ and positive for $d_{min} \ge 0.2875$, while the final residual jumps by a factor $3$--$5$ across the same threshold---the transition localizes the true minimal separation $0.272$ to within $6\%$, i.e.\ it recovers the geometry of the solution set. LdG block: the hinge is satisfied on \emph{both} sides of the edge, which instead manifests as a residual plateau ($^{\dagger}$the disagreement between the two seeds at $d_{min}=0.5125$ is the seed-dependent transition band around $0.5$ discussed in the text).}}
\label{tab:bisect}
\end{table}

\section{Theorems and Proofs}

\begin{Def}[Star Domain]
\label{defStarDomain}
Let $\Omega \subset \mathbb R^d$ be an open domain. $\Omega$ is called a star domain with center $x_0 \in \Omega$, if for all $y \in \Omega$ and $t \in [0,1]$ we have that
\[
    tx_0 +(1-t)y \in \Omega
\]
\end{Def}



\begin{lemma}
\label{lemmaForApprox}
    For $n \in \mathbb N$ let $N_n \subset \mathbb R^d$ be compact sets. Assume that $N_n \supset N_{n+1} $ for all $n \in \mathbb N$ and that $N := N_1 = \bigcup \limits _{n \in \mathbb N} N_n \subset\mathbb R^d $ is bounded. Assume that $h \in C(\Bar{N })$ and $h \vert {\bigcap \limits _{n \in \mathbb N} N_n} = 0$. Then
    \[
    \norm{h}_{\infty, N_n} \xrightarrow{n \to \infty} 0
    \]
\end{lemma}
\begin{proof}
    Assume $\exists C >0$ s.t. $\forall n \in \mathbb N: ~ \norm{h}_{\infty, N_n} > C $. Then we find a sequence of points $x_n \in N_n$, s.t. $|h(x_n)| = \sup \limits _{x \in N_n} |h(x)| > C$. $(x_n)_{n \in \mathbb N} \subset N$ is a bounded sequence. Bolzano--Weierstrass gives us a converging subsequence $(x_{n_k})_{k \in \mathbb N}$ with limit $x_0$. For every $m \in \mathbb N$ we have $x_{n_k} \in N_{n_k} \subset N_m$ whenever $n_k \ge m$, and $N_m$ is closed, so $x_0 \in N_m$; hence $x_0 \in \bigcap \limits _{n \in \mathbb N} N_n $. By continuity of $h$, $|h(x_0)| = \lim_{k \to \infty} |h(x_{n_k})| \geq C >0 $, which contradicts $h \vert_{\bigcap_n N_n} = 0$.
\end{proof}

\theoremApproxmain*

\begin{proof}
    Fix $\varepsilon >0$.
    We begin by defining for $\lambda \geq 0$
    \[ 
    A_ \lambda := \{ x \in \Omega: ~ dist(x, \partial \Omega) > \lambda  \}
    \]
    $A_\lambda$ has the following properties:
    \begin{itemize}
        \item $A_ \lambda$ is open for all $\lambda \geq 0$
        \item $\exists \lambda > 0$  s.t. $A _ \lambda \neq \emptyset$
        \item If $A_\lambda \neq \emptyset$, then $A _ { \lambda \over 2} \supsetneq A_ \lambda$
    \end{itemize}
    Now we observe that for a fixed $\lambda >0$ s.t. $A_\lambda \neq \emptyset$, if we define $A^n := A_{\lambda \over {2^n}}$ and $N_n :=  \Bar{\Omega} \setminus A^n$ (compact, since $\Omega$ is bounded), then $(N_n)_{n \in \mathbb N}$ fulfills the conditions of Lemma \ref{lemmaForApprox} above (with $\bigcap_n N_n = \partial \Omega$) and thus we can find a $n \in \mathbb N$ s.t. 
    $ \norm{h}_{\infty, N_n}<\varepsilon$.
    Now we fix this $n$ and define 
    \[ \begin{cases*}
        A:= A^n \\
        B:= A^{n+1}
    \end{cases*}\]
    \\
    \\
    Define the function $f(x) = {h(x) \over g(x)} $ for $x\in \Omega$. Note that $f$ is continuous on $\Omega$, but not necessarily continuous on $\Bar \Omega$. 
    Now we take a continuous function $ \varphi \in C( \Bar \Omega)$, s.t. $0 \le \varphi \le 1$, $\varphi = 1$ on $A$ and $\varphi = 0$ on $ \Bar \Omega \setminus B$. Define the function 
    \[f_ \varepsilon := f \cdot \varphi \in C_0( \Bar \Omega) = \{ g \in C(\Bar \Omega): ~ g \vert _ {\partial \Omega} = 0  \} \subset  C( \Bar \Omega). \]
    The following simple estimate concludes the proof:
    \begin{eqnarray*}
        \norm{f_\varepsilon g -h} _ {\Omega} &\leq& \underbrace{\norm{f_\varepsilon g -h}_{A}}_{\substack{=\norm{(f_\varepsilon  -f)g}_{A} = 0}} +  \norm{f_\varepsilon g -h}_{\Omega \setminus A}\notag\\
        &\leq&  \norm{f_\varepsilon g}_ {\Omega \setminus A} + \underbrace{ \norm{h}_{\Omega \setminus A}}_{\leq \varepsilon}\notag \\
        &\leq&   \norm{f \varphi g}_ {\Omega \setminus A} +  \varepsilon \notag\\
        &= &    \norm{f \varphi g}_ {B \setminus A} +  \varepsilon \notag \\
        & \leq & \norm{fg}_{B \setminus A} +  \varepsilon \notag\\
        &=& \norm h _{B \setminus A} +  \varepsilon \notag\\
        &\leq& \norm h _{ \Omega \setminus A} +  \varepsilon \leq 2 \varepsilon
    \end{eqnarray*}
\end{proof}

\section{DeepONets}\label{sec:DeepOnets}
\textit{Deep Operator Networks (DeepONets)} are a neural network framework for learning operators between function spaces. They are based on the classical approximation result for operators by \cite{ChenChenApproxOP}.
\begin{thm}[\cite{ChenChenApproxOP}]
    \label{DeepONetChenChenThm}
    Suppose that
    \begin{itemize}
        \item $\sigma: \mathbb R\to \mathbb R$ is continuous and non-polynomial activation function,
        \item $X$ is a Banach space,
        \item $K_1 \subset X$, $K_2 \subset \mathbb R^d$ are both compact,
        \item $V\subset C(K_1)$ is compact,
        \item $G: V \to C(K_2)$ be a continuous (non-linear) operator,
    \end{itemize}
    then for all $\varepsilon >0$ there exists $n,~ p, ~ m \in \mathbb N$, $c_i^k, ~ \xi_{ij}^k, ~ \theta_i^k ,~ \zeta _ k \in \mathbb R$, $w_k \in  \mathbb R^d $, $x_j \in K_1$ such that for all $u  \in V$ and $y \in K_2$ we have that
    \[
    \Big|  G(u)(y)- \sum \limits _{k=1}^p \sum \limits _{i = 1}^n c_i^k \sigma  \Big( \sum \limits _{j=1}^m \xi_{ij}^k u(x_j) + \theta_i ^k  \Big) \sigma(w_k \cdot y + \zeta_k) \Big| < \varepsilon.
    \]
    The fixed points $x_j$ are called \textbf{sensor points}.
\end{thm}

\cite{DeepONetsOriginal} reused this idea to introduce what we now know as DeepONets.
Let us now give a precise definition of DeepONets. This definition is borrowed from \cite{errorEstimationDeepONet}.
\begin{Def}
    Let $D \subset \mathbb R^d$ be a compact set and $m, p \in \mathbb N$. Define the following operators
    \begin{itemize}
        \item \textbf{Encoder:} Given a set of sensor points $(x_j)_{j= 1}^m \subset D$, we define the linear mapping $\begin{cases}\mathcal E: C(D) \to \mathbb R^m, \\
        \mathcal{E} (u) = (u(x_1),\dots, u(x_m))
        \end{cases}$
        \item \textbf{Approximator:} Given the evaluation of the sensor points, the approximator is a neural network, that maps $ \mathcal A : \mathbb R ^m \to \mathbb R^p $. The composition $\mathcal A \circ  \mathcal E $ is called \textbf{Branch Net}.
        \item \textbf{Reconstructor:} We introduce a neural network
        $\begin{cases}
        \tau : \mathbb R^d \to \mathbb R^{p+1}\\
        y \mapsto (\tau_0(y), \dots , \tau_p(y))
        \end{cases}$ and call it the \textbf{Trunk Net}. Then the reconstructor is defined as $\begin{cases}
            \mathcal R = \mathcal R_\tau : \mathbb R^p \to C(D),\\
            \alpha \mapsto  \tau_0 (\cdot) + \sum  \limits _{k = 1}^p \alpha _k \tau_k (\cdot)  
        \end{cases}$.
    \end{itemize}
    Then the composition of branch net and reconstructor, $\mathcal R \circ \mathcal A \circ \mathcal E$, is called a \textbf{DeepONet}.
\end{Def}

In the original paper \cite{DeepONetsOriginal}, the authors propose two different implementations of the approximator. The first is the \textit{unstacked DeepONet}, where the approximator consists of only one neural network, as in the above definition, and the second is the \textit{stacked DeepONet}, where the approximator consists of one neural network for each output feature. More precisely, we get
\[
\mathcal A _{\text{stacked}}(u_1,\dots,u_m) = (\mathcal{A}_1(u_1,\dots,u_m), \dots, \mathcal{A}_p(u_1,\dots,u_m))
\]
where each $\mathcal A_i: \mathbb R^m \to \mathbb R$ is a neural network.
Although this is the version used in Theorem \ref{DeepONetChenChenThm}, the authors outline the drawbacks in computational complexity of the stacked version and even show in their experimental findings that the unstacked version outperforms the stacked version in terms of total error. Furthermore, it can be easily shown that we also get universality for the unstacked version, by using the fact that neural networks are universal approximators.
Furthermore, for an extensive error analysis for DeepONets we refer to \cite{errorEstimationDeepONet}.

\end{document}